\documentclass[11pt]{amsart}
\usepackage{latexsym,amssymb,amsmath,hyperref,empheq}
\usepackage{amsthm, dsfont}
\usepackage{tikz, listings}
\usepackage[all]{xy}
\usepackage[short,nodayofweek]{datetime}
\usepackage[active]{srcltx}

\leftmargin  \leftmargini
\def\ZZ{{\mathbb Z}}

\def\RR{{\mathbb R}}
\def\CC{{\mathbb C}}

\def\F{{\mathcal F}}

\def\H{{\mathcal H}}

\def\U{{\mathcal U}}
\def\P{{\mathcal P}}
\def\N{{\mathcal N}}

\def\aaa{{\mathfrak a}}
\def\gg{{\mathfrak g}}
\def\hh{{\mathfrak h}}
\def\kk{{\mathfrak k}}
\def\pp{{\mathfrak p}}
\def\zz{{\mathfrak z}}

\def\Hop{\!\cdot\!}

\newcommand{\vol}{\operatorname{vol}}

\def\d{{\rm d}}

\DeclareMathOperator{\GL}{GL}

\DeclareMathOperator{\trace}{tr}
\DeclareMathOperator{\im}{Im}
\DeclareMathOperator{\re}{Re}
\DeclareMathOperator{\Sp}{Sp}

\DeclareMathOperator{\SL}{SL}

\theoremstyle{plain}
\newtheorem{thm}{Theorem}[section]
\newtheorem{cor}[thm]{Corollary}
\newtheorem{lem}[thm]{Lemma}

\newtheorem{prop}[thm]{Proposition}

\theoremstyle{definition}
\newtheorem{defn}[thm]{Definition}

\theoremstyle{remark}

\newtheoremstyle{Acknowledgements}
  {}
    {}
     {}
     {}
    {\bfseries}
    {}
     {.5em}
     {\thmname{#1}\thmnumber{ }\thmnote{ (#3)}}
\theoremstyle{Acknowledgements}

\date{\today, \currenttime} 

\begin{document}
\title[Holomorphic projection]{Holomrophic projection for $\Sp_2(\RR)$ -- The case of weight $(4,4)$}

\author{Kathrin Maurischat}
\address{\rm {\bf Kathrin Maurischat}, Mathematisches Institut,
   Heidelberg University, Im Neuenheimer Feld 288, 69120 Heidelberg, Germany }
\curraddr{}
\email{\sf maurischat@mathi.uni-heidelberg.de}
\thanks{This work was supported by the European Social Fund.}

\begin{abstract}
We define non-holomorphic Poincar\'e series of exponential type for symplectic groups $\Sp_m(\RR)$ and continue them analytically in case $m=2$ for the small weight $(4,4)$.
For this we construct certain Casimir operators and study the spectral properties of their resolvents on $L^2(\Gamma\backslash \Sp_2(\RR))$.
Using the holomorphically continued Poincar\'e series, the holomorphic projection is described in terms of Fourier coefficients using  Sturm's operator.

\end{abstract}

\maketitle
\setcounter{tocdepth}{1}
\tableofcontents


\section*{Introduction}
The explicit description of the projection of functions  to their holomorphic part  is important for several parts of the theory of automorphic functions.
For example it revealed the true role of mock theta functions, and one can use it to  describe the central values of Rankin $L$-functions.
The classical approach  for symplectic groups to holomorphic projection  is by Sturm-type arguments (\cite{sturm}).
Coarsely, given  Fourier coefficients of a function $F$, averaging over their imaginary parts yields Fourier coefficients of a holomorphic function.
This is proven by unfolding the inner product of $F$ against a system of Poincar\'e series.
Sturm~\cite{sturm} invented this method for genus one and high weight.
Panchishkin~\cite{panchishkin} established an approach for arbitrary genus $m$ and high scalar weights $\kappa>2m$.
But it is the case of low weight $\kappa<2m$ which is of arithmetical interest, e.g. in \cite{gross-zagier} the interesting weight for genus one is two.
We develop a method which establishes the holomorphic projection by Sturm's operator for genus two and scalar weight $\kappa=4$. The method also applies to weight $\kappa=3$.
But there other phenomenons arise which do not lead to a proper description of the holomorphic projection (see~\cite{phantoms}).

The description of cusp forms by Poincar\'e series is first systematically presented in~\cite{petersson}.
For $\SL_2(\RR)$ Neuenh\"offer 
\cite{neunhoeffer} establishes the case of weight zero. He introduces essentially two types of Poincar\'e series
\begin{equation*}
 \sum_{\gamma\in\Gamma_\infty\backslash\Gamma}f_\tau(\gamma\Hop z,s)
\end{equation*}
 involving
\begin{equation}\label{p-reihen-erster-art}
 f_\tau(z,s)\:=\: \lvert z\rvert^\tau (1-\lvert z\rvert^2)^s\:,
\end{equation}
respectively
\begin{equation}\label{p-reiehn-zweiter-art}
 f_\tau(z,s)\:=\:e^{2\pi i\tau z} \im(z)^s\:.
\end{equation}
Panchishkin~\cite{panchishkin} studies two-variable Poincar\'e series  
generalizing type~(\ref{p-reihen-erster-art}) above. 
Klingen~\cite{Klingen} uses Poincar\'e series  for arbitrary genus $m$ generalizing type
(\ref{p-reiehn-zweiter-art}) in an  obvious way.
We use Poincar\'e series $\mathcal P_\tau(g,s_1,s_2)$ of exponential type and scalar weight $\kappa$ on the symplectic group $\Sp_m(\RR)$ which are generalizations of the latter choosing
\begin{equation*} f_\tau(g,s_1,s_2)\:=\:
e^{2\pi i\trace(\tau z)}\frac{\trace(\tau\im z)^{s_1}\det(\im z)^{s_2}}{j(g,i)^\kappa}\:,
\end{equation*}
 for complex variables $s_1,s_2$ and $z=g\cdot i$ in the Siegel halfspace. We have the following 
convergence result (Corollary~\ref{cor_konvergenzbereich_u_v}).
\begin{thm}
 These Poincar\'e series $\mathcal P_\tau(g,s_1,s_2)$
converge absolutely and  uniformly on compact sets within
\begin{equation*}
 \{(s_1,s_2)\in\CC^2\mid \re(2s_2+\kappa)>2m\textrm{ and } \re(\frac{2}{m}s_1+2s_2+\kappa)>2m\}\:.
 \end{equation*}
 There, they belong to $L^2(\Gamma\backslash G)$.
\end{thm}
For high weight $\kappa>2m$ these Poincar\'e series are cusp forms at their point of holomorphicity $(s_1,s_2)=(0,0)$.
Most part of this work is devoted to their analytic continuation for low weight $\kappa\leq 2m$.
For this we apply resolvents of certain Casimir operators.
In case of genus one this approach (see Appendix) is a well-studied    application of the theory of Eisenstein series due to Roelcke \cite{roelcke1}-\cite{roelcke4}.
For genus two there is a series of technical difficulties we have to manage. 
We apply Langlands' theory of Eisensteins series (section~\ref{section_spectrum}) for $L^2(\Gamma\backslash \Sp_2(\RR))$ to localize the spectral roots (section~\ref{sec_resolvents})
of the two Casimir operators
\begin{eqnarray*}
 D_+(u,\Lambda)&=&\prod_{\alpha \textrm{ long root}}\bigl(\check{\alpha}(\Lambda)-u\bigr)\:,\\
D_-(v,\Lambda)&=&\prod_{\alpha \textrm{ short root}}\bigl(\check{\alpha}(\Lambda)-v\bigr)\:.
\end{eqnarray*}
Here $\Lambda\in\aaa_\CC^\ast$ is a spectral parameter (infinitesimal character) encoded by some Cartan subalgebra $\aaa$ of the symplectic lie algebra  and $\check{\alpha}$ 
denotes the coroot of the root $\alpha$. The complex variables $u$ and $v$ are affine linear transforms of $s_1=\frac{1}{2}(v-2u-1)$ and $s_2=\frac{1}{2}(u-(\kappa-m))$.
The existence of the resolvents $R_+(u)$ and $R_-(v)$ of $D_+(u)$ and $D_-(v)$, respectively, as meromorphic functions is restricted by the occurence of continuous spectral components 
in $L^2(\Gamma\backslash \Sp_2(\RR))$. We have 
(Propositions~\ref{prop_discrete_spec}, \ref{prop_cont_spec_2-dim}, \ref{prop_cont_spec_1-dim}):
\begin{thm}
 The resolvent $R_+(u)$ is a meromorphic function  on $\re u>\frac{1}{2}$.
 The resolvent $R_-(v)$ is a meromorphic function  on $\re v>1$.
\end{thm}
The operators $D_+(u)$ and $D_-(v)$ are constructed such that their resolvents applied to the Poincar\'e series give the meromorphic continuation of the latter in the 
direction of $u$ and of $v$,
respectively. This involves a number of vast computations solved with the  computer algebra system Magma. By this and some simple consequences of the  theory of 
Eisenstein series we get analytic continuation to the point of holomorphicity (Theorem~\ref{thm_continuation_kappa_2m}) in case
of scalar weight $\kappa=4$.
\begin{thm}
  Let $m=2$ and $\kappa=4$. 
The $L^2$-limit
\begin{equation*}
 \mathcal P_\tau(\cdot,0,0)\:=\:\lim_{s_2\to 0} \mathcal P_\tau(\cdot,0,s_2)
\end{equation*}
exists in $L^2(\Gamma\backslash \Sp_2(\RR))$. It has got a holomorphic $C^\infty$-representative.
\end{thm}
Having continued the Poincar\'e series holomorphically, holomorphic projection is immediate (Theorem~\ref{satz_holomorphe_projektion}).
\begin{thm}
 The Sturm operator establishes the holomorphic projection in case of genus $m=2$ and scalar weight $\kappa\geq 4$.
\end{thm}
Concerning generalizations, the case of arbitrary symplectic groups seems out of reach at the moment, as the number of necessary calculations grows exponentially.
But the shapes of $D_+(u,\Lambda)$ and $D_-(v,\Lambda)$ lead to the suspection that  some generalization for special orthogonal groups $\mathop{SO}(2,2n+1)$ should be possible.
This is upcoming work.

\section{Casimir elements}\label{section_Casimir}
Let $G=\Sp_m(\RR)$ be the symplectic group of genus $m$. 
Later on  
we will restrict to the case  $m=2$. 
We realize  $G$ as the group of those $g\in M_{m,m}(\RR)$ satisfying $g'Wg=W$ for 
\begin{equation*}
 W=\begin{pmatrix}0&-E_m\\E_m&0\end{pmatrix}.
\end{equation*}
We have the usual action of $G$ on the Siegel halfspace $\H$, for $g=\begin{pmatrix}a&b\\c&d\end{pmatrix}\in G$,
\begin{equation*}
 g\Hop z=(az+b)(cz+d)^{-1} \:.
\end{equation*}
Let $K$ be the stabilizer of $i=iE_m\in\H$, thus  $K$ is a maximal compact subgroup of 
$G$. We denote by
\begin{displaymath}
 g\:\mapsto \:g\Hop i=:z=x+iy
\end{displaymath}
the obvious isomorphism of $G/K$ to $\H$.
Let $\mathcal F$ be the Siegel fundamental domain for the action of $\Sp_m(\ZZ)$ on $\H$.
We define the function $J:G\times\H\to\GL_2(\CC)$, $J(g,z)=cz+d$, and the factor of automorphy 
$ j(g,z)= \det (J(g,z))$.
The  constraint of $J:=J(\cdot,i)$ to $K$ defines the isomorphism of $K$ to the unitary group $U_m$.
By Iwasawa decomposition, every element $g\in G$ can be written as $g=pk$, where $k\in K$ and $p$ is parabolic,
\begin{equation*}
 p\:=\:\begin{pmatrix}
   T&U\\0&T'^{-1}
  \end{pmatrix}.
\end{equation*}
Here $T$ has lower triangular shape. 
Choosing all its diagonal elements $t_1,\dots,t_m$  to be positive, $T$ is uniquely determined by $g$, and $z=UT'+iTT'$.
Let $\mathfrak g$ be the Lie algebra of $G$.
We have the matrix realization of $\mathfrak g_\CC\subset M_{2m,2m}(\mathbb C)$ consisting of those $g$ satisfying
$g'W+Wg=0$.
Then $\mathfrak g_\CC=\mathfrak p_+\oplus\mathfrak p_-\oplus\mathfrak k_\CC$,
where $\mathfrak k_\CC$ is the Lie algebra of $K$ given by the matrices satisfying
\begin{equation*}
 \begin{pmatrix}A&S\\-S&A\end{pmatrix}\:, \quad A'=-A\:, \quad S'=S\:,
\end{equation*}
and
\begin{equation*}
 \mathfrak p_\pm=\left\{\begin{pmatrix}X&\pm iX\\\pm iX&-X\end{pmatrix},\quad X'=X\right\}.
\end{equation*}
Let $e_{kl}\in M_{m,m}(\mathbb C)$ be the elementary matrix having entries $(e_{kl})_{ij}=\delta_{ik}\delta_{jl}$ 
and let $X^{(kl)}=\frac{1}{2}(e_{kl}+e_{lk})$.
 The elements $(E_\pm)_{kl}=(E_\pm)_{lk}$ of $\mathfrak p_\pm$ are defined to be those corresponding to $X=X^{(kl)}$, 
$1\leq k,l\leq m$. 
Then $(E_\pm)_{kl}$, $1\leq k\leq l\leq m$ form a basis of $\mathfrak p^\pm$.
A basis of $\mathfrak k_\CC$  is given by $B_{kl}$, for $1\leq k,l\leq m$, where  $B_{kl}$ corresponds 
to $A_{kl}=\frac{1}{2}(e_{kl}-e_{lk})$ and $S_{kl}=\frac{i}{2}(e_{kl}+e_{lk})$.
For abbreviation, let $E_\pm$ be the matrix having entries $(E_\pm)_{kl}$. 
Similarly, let $B=(B_{kl})_{kl}$ be the matrix with entries $B_{kl}$ and let $B^\ast$ be its transpose 
having entries $B_{kl}^\ast=B_{lk}$.
Thus, $E_+$, $E_-$, $B$ and $B^\ast$ are matrix valued matrices. 
Taking formal traces of them and their formal products, e.g. $\trace(E_+E_-)$, such traces are not invariant under cyclic permutations of their arguments.
%
For $\gg=\mathfrak{sp}_m(\RR)$ the center $\zz_\CC$ of the universal enveloping Lie algebra $\mathfrak U(\mathfrak g_\CC)$
 is generated by $m$ elements.
For any basis $\{X_i\}$  of  $\mathfrak g_\CC$  let $\{X_i^\ast\}$ 
be its dual with respect to the nondegenerate bilinear form $\mathcal B$  on $\mathfrak g_\CC$,
\begin{equation}\label{def_bilinearform}
 \mathcal B(g,h)\:=\:\frac{1}{2}\trace(g\cdot h)\:.
\end{equation}
The Killing form is  given by $4(m+1)\mathcal B$. Then the elements
\begin{equation*}
 D_r\:=\:\sum_{i_1,\dots,i_r}\trace(X_{i_1}\cdots X_{i_r})X_{i_1}^\ast\dots X_{i_r}^\ast
\end{equation*}
are easily seen to belong to the center  of the universal enveloping algebra  and are 
independent of the chosen basis. Here $\trace(X_{i_1}\cdots X_{i_r})$ denotes the trace of the matrix product $X_{i_1}\cdots X_{i_r}$.
With respect to $\mathcal B$, we get the  dual basis
$(E_\pm)_{kl}^\ast= \frac{1}{1+\delta_{kl}}(E_\mp)_{kl}$ as well as
$B_{kl}^\ast=B_{lk}$ for all $k,l$.
\begin{prop}\label{Prop_Casimir_2}\cite{CasimirPaper},\cite{weissauersLN}
In terms of the basis above,
\begin{equation*}
 D_2=\trace(E_+E_-)+\trace(E_-E_+)+\trace(BB)+\trace(B^\ast B^\ast)\:,
\end{equation*}
\begin{eqnarray*}
 D_4 &=& \trace(E_+E_-E_+E_-)+\trace(E_-E_+E_-E_+)+\trace(BBBB)+\trace(B^\ast B^\ast B^\ast B^\ast)\\
&& +\sum_{\zeta\in Z_4}\bigl(\trace(\zeta(E_+E_-B^\ast B))+\trace(\zeta(E_-E_+B B^\ast))+\trace(\zeta(E_+B E_-B^\ast))\bigr),
\end{eqnarray*}
where $Z_4$ is the group of cyclic permutations of four elements.
\end{prop}
For applications, we  get the following reformulations:
\begin{cor}\label{cor_casimir_operatoren}
Let $C_1:=\frac{1}{2}D_2$ and $C_2:=\frac{1}{2}D_4$. Then
\begin{equation*}
 C_1=\frac{1}{2}(\trace(E_+E_-)+\trace(E_-E_+))+\trace(BB)\:,
\end{equation*}
\begin{eqnarray*}
 C_2&=& \frac{1}{2}\bigl(\trace(E_+E_-E_+E_-)+\trace(E_-E_+E_-E_+)+\trace(B^4)+\trace((B^\ast)^4)\bigr)\\
&&+2\bigl(\trace(E_+E_-B^\ast B^\ast)+\trace(E_-E_+B B)\bigr)\\
&&-\sum_{i,j,k,l}\{(E_+)_{kl},(E_-)_{ij}\}B_{kj}B_{li}\\
&&+\frac{(m+1)^2}{2}(\trace(E_+E_-)+\trace(E_-E_+))\:.
\end{eqnarray*}
where 
\begin{equation*}
 \frac{1}{2}(\trace(E_+E_-)+\trace(E_-E_+))=\trace(E_+E_-)+(m+1)\trace(B)\:,
\end{equation*}
\begin{eqnarray*}
&&\hspace*{-10mm}\frac{1}{2}(\trace(E_+E_-E_+E_-)+\trace(E_-E_+E_-E_+))\\ &&=\trace(E_+E_-E_+E_-)+\frac{1}{2}(\trace(E_+E_-)
+\trace(E_-E_+))\trace(B)\\
&&\quad+\frac{m+2}{2}\bigl(\trace(E_+E_-B^\ast)+\trace(E_-E_+B)\bigr)\:.
\end{eqnarray*}
For $m=2$, the Casimir elements $C_1, C_2$ generate the center $\mathfrak z_\CC$ of the universal enveloping algebra.
\end{cor}
\begin{proof}[Proof of Corollary~\ref{cor_casimir_operatoren}]
The formulae for the traces are obtained by rearranging. Similarly the formulae 
for $C_1,C_2$ follow  by rearranging those of Proposition~\ref{Prop_Casimir_2}. For $m=2$, we will see in  Prop.~\ref{prop_harish-chandra},
that $\mathfrak z_\CC$ is generated by $C_1,C_2$.
\end{proof}
\subsection{Harish-Chandra homomorphism}\label{section_harisch-chandra}
The following wellknown results on the Harish-Chandra homomorphism and $K$-types  are included  in order
to fix the precise values for Casimir operators needed later on, which indeed depend on the choices.
The Harish-Chandra homomorphism is  described as follows (see for example \cite[IV.~7]{KnappVogan}). 
Take any Cartan subalgebra $\hh_\CC$ of $\gg_\CC$ and fix a system of positive roots $\Delta^+$ of $\gg_\CC$ for 
$\hh_\CC$ and let $\delta$ be half the sum of positive roots.
Define $\P=\sum_{\gamma\in\Delta^+}\U(\gg_\CC)\gg_{\gamma}$ and $\N=\sum_{\gamma\in\Delta^+}\gg_{-\gamma}\U(\gg_\CC)$.
Then we have
\begin{equation*}
 \U(\gg_\CC)\:=\: \U(\hh_\CC)\oplus (\P+\N)\:.
\end{equation*}
Let $p_+:\U(\gg_\CC)\to\U(\hh_\CC)$ be the projection with respect to this decomposition.
Let $\tau_+:\U(\hh_\CC)\to\U(\hh_\CC)$ be given on $\hh_\CC$ by 
\begin{equation*}
 \tau_+(h)\:=\:h-\delta(h)
\end{equation*}
and on $\U(\hh_\CC)$ by algebraic continuation.
Then the Harish-Chandra homomorphism
 \begin{equation*}
 \gamma\:: \:\U(\gg_\CC)\:\longrightarrow\:\U(\hh_\CC)
 \end{equation*}
is  $\tau_+\circ p_+$. Restricted to $\zz_\CC$ this is an isomorphism
\begin{equation*}
 \gamma\::\:\zz_\CC\:\tilde\longrightarrow\: \U(\hh_\CC)^W,
\end{equation*}
which is independent from the chosen positive system $\Delta^+$. Here $W$ denotes the Weyl group $W(\gg_\CC,\hh_\CC)$.
For explicit formulae in case of genus two, we choose
\begin{equation*}
 \hh_\CC \:=\:\CC B_{11}+\CC B_{22} \:\subset\: \kk_\CC\:,
\end{equation*}
which is a Cartan subalgebra  for both, $\kk_\CC$ and $\gg_\CC$. 
Let $\Delta^+$ be the set of positive roots for $\hh_\CC$ such that their root spaces belong to $\CC B_{12}+\pp^-$. 
Writing $\Lambda=(\Lambda_1,\Lambda_2)$ for $\Lambda \in\hh_\CC^\ast$, where $\Lambda_j=\Lambda(B_{jj})$, 
these root spaces are 
\begin{eqnarray*}
 \gg_{(1,-1)}=\CC B_{12}, & \gg_{(2,0)}=\CC (E_-)_{11}\:,\\
\gg_{(1,1)}=\CC (E_-)_{12}, & \gg_{(0,2)}=\CC (E_-)_{22}\:.
\end{eqnarray*}
Half the sum of positive root  is
\begin{equation*}
 \delta\:=\:\delta_G\:=\:\frac{1}{2}\sum_{\Lambda\in\Delta^+}\Lambda=(2,1)\:,
\end{equation*}
while $\delta_K=\frac{1}{2}(1,-1)$, and 
\begin{eqnarray*}
 \P&=& \U(\gg_\CC)\pp^- + \U(\gg_\CC)B_{12}\:,\\
\N&=& \pp^+\U(\gg_\CC) + B_{21}\U(\gg_\CC)\:.
\end{eqnarray*}
Next we compute the images of $C_1$ and $C_2$ under $\gamma$.
Using Corollary~\ref{cor_casimir_operatoren} and the Lie bracket relations, we get
\begin{equation*}
p_+(C_1)\:=\:p_+\bigl(\trace(B^2)+(m+1)\trace(B)\bigr)\:,
\end{equation*}
where 
\begin{equation*}
 p_+(\trace(B^2))\:=\:\sum_j B_{jj}^2+(B_{11}-B_{22})\:.
\end{equation*}
And
\begin{eqnarray*} 
 &&p_+\left(\frac{1}{2}(\trace(B^4)+\trace(B^{\ast4}))\right)\:=\:\\
 &&\hspace*{10mm}B_{11}^4+B_{22}^4+(B_{11}-B_{22})\left(2(B_{11}^2+B_{22}^2+B_{11}B_{22})+B_{11}-B_{22}+1\right)\:,
 \end{eqnarray*}
 \begin{equation*}
   p_+(\frac{1}{2}\trace(E_-E_+E_-E_+))\:=\:8(B_{11}^2+B_{22}^2)+5(B_{11}+B_{22})^2+8(B_{11}-B_{22})\:,
 \end{equation*}
 \begin{eqnarray*} 
 &&p_+(2\trace(E_-E_+BB))\:=\:8(B_{11}^3+B_{22}^3)+2(B_{11}+B_{22})(B_{11}^2+B_{22}^2)\\
 &&\hspace*{40mm}+2(B_{11}-B_{22})(9B_{11}+5B_{22}+4)\:,\\
 && p_+(\sum_{i,j,k,l}\{(E_+)_{kl},(E_-)_{ij}\}B_{kj}B_{li})\:=\:\\
 &&\hspace*{28mm}4(B_{11}^3+B_{22}^3)+2(B_{11}+B_{22})B_{11}B_{22}+5(B_{11}^2-B_{22}^2)\:.
\end{eqnarray*}
Applying   $\tau_+(B_{jj})=B_{jj}-(m+1-j)$ we receive
\begin{prop}\label{prop_harish-chandra}
 Let $m=2$. The images of the Casimir elements under the Harish-Chandra homomorphism are
\begin{eqnarray*}
 \gamma(C_1)&=&B_{11}^2+B_{22}^2-5\:,\\
\gamma(C_2)&=&B_{11}^4+B_{22}^4-17+3\gamma(C_1)\:.
\end{eqnarray*}
As $\gamma(C_2)$ is not a multiple of $\gamma(C_1)$, they generate $\U(\hh_\CC)^W$.
\end{prop}
\begin{cor}\label{harish-chandra}
 Let $m=2$ and $\Lambda=(\Lambda_1,\Lambda_2)\in\hh_\CC^\ast$. Then
\begin{eqnarray*}
 \Lambda(C_1)\:=\:\Lambda(\gamma(C_1))&=&\Lambda_{1}^2+\Lambda_{2}^2-5\:,\\
\Lambda(C_2)\:=\:\Lambda(\gamma(C_2))&=&\Lambda_{1}^4+\Lambda_{2}^4-17+3\Lambda(C_1)\:.
\end{eqnarray*}
\end{cor}
By Bezout's theorem, the Casimir elements have at most eight zeros $\Lambda$ in common.
These obviously are $(\pm1,\pm2)$, $(\pm2,\pm1)$, the Weyl group conjugates of $\Lambda=(2,1)$.
Cor.~\ref{harish-chandra} is independent of the chosen Cartan subalgebra $\hh$ in the sence that any isometric isomorphism to a second Cartan subalgebra $\tilde \hh$ will produce the same
formulae. Especially, if we choose the diagonal subalgebra $\mathfrak a=\tilde\hh$ then Cor.~\ref{harish-chandra} remains true with respect to the coordinate functions $\Lambda_1,\Lambda_2$.
\subsection{Representations of scalar $K$-type}\label{heighest_weight_representations}
It is wellknown (see~\cite{maass}) that $\trace((E_+E_-)^n)$ are invariant differential operators for fixed
$K$-type.
However, the Casimir operators  are globally defined.  
Here we fix the connection for scalar weight  $(\kappa,\dots,\kappa)$.

\begin{lem}\label{Lemma_j_fuer_U}
Let $\pi$ be a representation of $K$ of highest weight $(\kappa,\dots,\kappa)$.
Then the actions  of the basis elements $B_{kl}$, $1\leq k,l\leq m$, of  $\mathfrak k_\CC$ is given by
\begin{equation*}
\pi(B_{kl})\:=\:\kappa\cdot\delta_{kl}.
\end{equation*}
\end{lem}
\begin{proof}
Let $\pi=(\kappa,\dots,\kappa)$ be irreducible. Then $V_\pi$ has dimension one and highest and lowest weight vectors coinside.
Under $dJ$ the element $B_{kl}$ is mapped to $A_{kl}-iS_{kl}=e_{lk}$.
As $\exp(te_{lk})=E_m+te_{lk}$ for $k\not=l$ respectively $\exp(te_{kk})=E_m+(e^{t}-1)e_{kk}$ is upper 
or lower triangular, we get $\pi(B_{kl})=\frac{d}{dt}\pi(\exp(te_{lk}))\mid_{t=0}=\kappa \delta_{kl}\frac{d}{dt}e^{t}
\mid_{t=0}=\kappa\delta_{kl}$.
\end{proof}
\begin{prop}\label{Prop_Casimir2_Operation}
Let $\pi$ be a smooth representation of $G$. Then the action of the Casimir elements on
its $K$-type $(\kappa,\dots,\kappa)$ are given by
\begin{equation*}
 \pi(C_1)\:=\: \pi(\trace(E_+E_-))-\kappa m(m+1-\kappa)
\end{equation*}
and
\begin{eqnarray*}
 \pi(C_2)&=&\pi(\trace(E_+E_-E_+E_-))+m\kappa^4\\
&&+((m+1)^2-2\kappa (m+1)+2\kappa^2)\bigl(\pi(\trace(E_+E_-))-\kappa m(m+1)\bigr)\:.
\end{eqnarray*}
\end{prop}
\begin{proof}[Proof of Proposition~\ref{Prop_Casimir2_Operation}]
Apply Lemma~\ref{Lemma_j_fuer_U}  to 
Corollary~\ref{cor_casimir_operatoren}.
The result on $C_1$ is due to~\cite[Chapter~4]{weissauersLN}. 
\end{proof}

\begin{prop}\label{weissauers_bem_ueber_K-typ}\cite[Theorem~1.1]{zhu}
 If $\pi$ and $\pi'$ are representations  of $G$ of the same infinitesimal character containing the same scalar $K$-type, then $\pi$ and $\pi'$ are isomorphic.
\end{prop}

\section{On the  spectral decomposition of $L^2(\Gamma\backslash G)$}\label{section_spectrum}
Let $G=\Sp_2(\RR)$ be the symplectic group of genus two. 
Let $\Gamma$ be any subgroup of finite index in the full modular group 
$\Sp_2(\ZZ)$ containing 
  the group 
\begin{equation*}
 \Gamma_\infty\:=\:\{\begin{pmatrix}
  \pm E_2&\ast\\0&\pm E_2 
 \end{pmatrix}\in\Sp_2(\ZZ)\}
\end{equation*}
of translations.
The group $G$ acts on $L^2(\Gamma\backslash G)$ by right translations, and this $G$-action
 comes along with an action of the universal enveloping algebra 
$\mathfrak U(\gg_\CC)$ on $\mathcal C^\infty$-vectors.
As $\Gamma\backslash G$ isn't compact, the spectrum of $L^2(\Gamma\backslash G)$ contains  continuous parts.
We need some knowledge of the spectral decomposition and extract  this out of Langlands' theory of Eisenstein series~\cite{langlands}. 
(See also~\cite{konno},\cite{weissauersLN}.)

As the action of the center $\zz_\CC$ of $\mathfrak U(\gg_\CC)$ commutes with that of $G$ and $\mathfrak U(\gg_\CC)$, 
it acts by scalars on  irreducible  components.
Via the Harish-Chandra homomorphism these scalars are determined by the infinitesimal character $\Lambda\in\mathfrak a_\CC^\ast$ of the representation.
In general, the action of some Casimir element $C$ on any spectral component parametrized by $\Lambda$ is  given by applying the
Harish-Chandra homomorphism to $C$. In case of a $1$- or $2$-dimensional  parametrization (i.e. in case of a component
of the continuous spectrum) this involves a $1$- or $2$-dimensional integral of (a residue of) an Eisenstein series against
$\Lambda$.


Let $\mathfrak a$ be the  split component of the Borel subgroup 
\begin{equation*}
 B\:=\:\left\{\begin{pmatrix} T&X\\0&T'^{-1}\end{pmatrix}\mid T \textrm{ upper triangular }\right\}\:\subset\: G\:.
\end{equation*}
Identifying $\aaa_\CC^\ast$ with $\CC^2$ by choosing Euclidean coordinates $\Lambda=(\Lambda_1,\Lambda_2)$, 
the system of positive roots corresponding to $B$ is
\begin{equation*}
 \Sigma^+:=\{\alpha_1=(0,2),\alpha_2=(1,-1),\alpha_1+\alpha_2,\alpha_1+2\alpha_2\} \subset \mathfrak a_\CC^\ast\:.
\end{equation*}
Let $\delta=(2,1)$ be half the sum of positive roots.
The Weyl group $W$ of $G$ acts on $\aaa_\CC^\ast$. It is generated by the simple reflections $s_{\alpha_1}$ and $s_{\alpha_2}$. 
We have $B=NAM$, where $N$ is the uniponent radical normalized by $B$, the diagonal torus $A$ has Lie algebra $\aaa$, and $M\cong Z_2\times Z_2$ 
is   finite. Correspondingly, for $g\in G$ we have $g=namk$, where $k=k(g)\in K$ and $a=a(g)\in A$ etc.
Let
\begin{equation*}
 E_B(g,\phi,\Lambda)\:=\:\sum_{\gamma\in(\Gamma\cap B)\backslash\Gamma}a(\gamma g)^{\delta+\Lambda}\phi(\gamma g,\Lambda)
\end{equation*}
be an Eisenstein series  for $B$. 
Here the function $\phi$ belongs to a space $V$ such that for all $g\in G$ the function $\phi(gk^{-1})$ is of the same $K$-type and  the function $\phi(mg)$ belongs to 
a simple admissible subspace $V_M\subset L_0^2((\Gamma\cap M)\backslash M)$ (so $\phi$ is any function on the finite quotient).
The Eisenstein series converges absolutely in the cone
\begin{equation*}
 I\::=\:\{\Lambda\in\aaa_\CC^\ast\vert\langle \Lambda,\gamma\rangle>\langle \delta,\gamma\rangle,\gamma\in\Sigma^+\}\:.
\end{equation*}
For a simple root $\alpha_i$ let $^\bullet \aaa =\ker(\alpha_i)\subset\aaa$ and define
 $^\dagger\aaa$ by $\aaa=\:^\bullet\aaa\perp\:^\dagger\aaa$.
For any $\Lambda\in\aaa_\CC^\ast$, let $\Lambda=\:^\bullet\Lambda+\:^\dagger\Lambda$, where 
$^\bullet\Lambda=\Lambda\vert_{^\bullet \aaa_\CC}$ and $^\dagger\Lambda=\Lambda\vert_{^\dagger \aaa_\CC}$ are continued
to $\aaa_\CC$ by zero. 
Especially, $^\bullet\alpha_i=0$, so $\CC\alpha_i=(^\dagger\aaa_\CC)^\ast$.
Let $B\subset\:^\bullet\!P$ be the standard parabolic subgroup of $G$ with split component $^\bullet\aaa$ and corresponding
decomposition $^\bullet\!P=\:^\bullet\!N\:^\bullet\!A\:^\bullet\!M$. Here $^\bullet\!M$ is either $\SL_2(\RR)\times Z_2$ (Klingen) or 
$\SL_2(\RR)\rtimes Z_2$ (Siegel). 
There is a parabolic subgroup $^\dagger\!P$ of $^\bullet\!M$ corresponding to $B$,
\begin{equation*}
 ^\dagger\!P\:=\:^\bullet\!N\backslash (^\bullet\!M\:^\bullet\!N\cap B)\:\subset \:^\bullet\!M\:.
\end{equation*}
The split component of $^\dagger\!P$ can be identified with $^\dagger\!\aaa$.
The identity
\begin{equation}\label{E-reihen-zerlegung}
 E_B(g,\phi,\Lambda)\:=\:\sum_{\gamma\in(\Gamma\cap\:^\bullet\!P)\backslash \Gamma}\:^\bullet\!a(\gamma g)^
{^\bullet\!\delta+\:^\bullet\!\Lambda}
\:^\bullet\!E(\gamma g,\phi,\:^\dagger\!\Lambda)
\end{equation}
holds for $\Lambda\in I$, if
\begin{equation*}
 ^\bullet\!E(g,\phi,\:^\dagger\!\Lambda)\::=\:\sum_{\bar\gamma\in(\Gamma\cap \:^\dagger\!P)\backslash 
(\Gamma\cap\:^\bullet\!M)}\:
^\dagger a(\bar\gamma g)^{^\dagger\!\delta+\:^\dagger\!\Lambda}\phi(\bar\gamma g)
\end{equation*}
is an Eisenstein series for $^\bullet\!M$.
Whenever the inner Eisenstein series is defined, the series (\ref{E-reihen-zerlegung}) converges 
on the convex hull $\mathcal C(I\cup s_{\alpha_i}(I))$.
As $s_{\alpha_i}(\:^\bullet\!\Lambda)=\:^\bullet\!\Lambda$, the initial terms of the scattering operator $M(s,\Lambda)$, $s\in W$,
are given by
\begin{equation*}
 M(s_{\alpha_i},\Lambda)\:=\:M(s_{\alpha_i},^\dagger\!\Lambda)\:,
\end{equation*}
where $M(s_{\alpha_i},^\dagger\!\Lambda)$ is identified with the scattering operator belonging to $^\bullet\!E$.
Especially, $M(s_{\alpha_i},^\dagger\!\Lambda)$ is meromorphic in the complex variable $^\dagger\!\Lambda$ with only a finite 
number of poles $^\dagger\!\Lambda=c$, where $0\leq c\leq 1$.
Accordingly, the Eisenstein series $E_B(g,\phi,\Lambda)$ is meromorphically continued to $\mathcal C(I\cup s_{\alpha_i}(I))$
with a finite number of pole hyperplanes $H_{\alpha_i}(c)$, with $0\leq c\leq 1$,
\begin{equation*}
 H_{\alpha_i}(c)\:=\:
\{\Lambda\in\aaa_\CC^\ast\vert \langle \Lambda,\check{\alpha_i}\rangle=c\}\:=\:\{\frac{c}{2}\alpha_i+(^\bullet\aaa_\CC)^\ast\} \:.
\end{equation*}
Here $\check{\alpha_i}$ is the coroot of $\alpha_i$.
The Eisenstein series $E_B(g,\phi,\Lambda)$ and the scattering operator $M(s,\Lambda)$ have
meromorphic continuation to $\aaa_\CC^\ast$ and satisfy the functional equations
\begin{eqnarray*}
 M(ts,\Lambda)&=& M(t,s\Lambda)M(s,\Lambda)\:,\\
E_B(g,\phi,\Lambda)&=& E_B(g,M(s,\Lambda)\phi,s\Lambda)\:.
\end{eqnarray*}
So $E$ and $M$ only have a finite number of pole hyperplanes given by the images of
$H_{\alpha_1   }(c)$ and $H_{\alpha_2}(c)$ under the Weyl group,
\begin{equation*}
 H_{\alpha_1+\alpha_2}(c)\:=\: s_{\alpha_1} H_{\alpha_2}(c)\:,\:\quad H_{\alpha_1+2\alpha_2}(c)\:=\: s_{\alpha_2} H_{\alpha_1}(c)\:.
\end{equation*}
%
Especially, the continuations are holomorphic on the cones $s(I)$, $s\in W$.
\begin{figure}
\begin{tikzpicture}
 \draw [->](0,-4.4) -- (0,4.4);
 \draw (-0.25,4.4) node {$\Lambda_2$};
 \draw [->,style=thick] (0,0)--(0,1);
 \draw (-0.25,0.9) node {$\alpha_1$};
\draw [->](-4.4,0) -- (4.4,0);
\draw (4.4,-0.4) node {$\Lambda_1$};
\draw [->,style=thick] (0,0)--(0.5,-0.5);
\draw (0.6,-0.6) node {$\alpha_2$};
\filldraw [lightgray] (4.4,0.5)--(1,0.5) -- (4.4,3.9);
\filldraw [lightgray] (4.4,-0.5)--(1,-0.5) -- (4.4,-3.9);
\filldraw [lightgray] (-4.4,0.5)--(-1,0.5) -- (-4.4,3.9);
\filldraw [lightgray] (-4.4,-0.5)--(-1,-0.5) -- (-4.4,-3.9);
\filldraw [lightgray] (0.5,4.4)--(0.5,1) -- (3.9,4.4);
\filldraw [lightgray] (0.5,-4.4)--(0.5,-1) -- (3.9,-4.4);
\filldraw [lightgray] (-0.5,4.4)--(-0.5,1) -- (-3.9,4.4);
\filldraw [lightgray] (-0.5,-4.4)--(-0.5,-1) -- (-3.9,-4.4);

\draw (3.2,1.3) node {$I$};
\draw (3.2,-1.3) node {$s_{\alpha_1} (I)$};
\draw (1.3,3.2) node {$s_{\alpha_2} (I)$};
\filldraw (1,0.5) circle (1pt);
\draw (1.4,0.65) node {$\delta$};

\draw (4.4,0.5)--(-4.4,0.5);
\draw (-4.4,0.7) node {$H_{\alpha_1}(1)$};
\draw  (-3.9,-4.4)--(4.4,3.9);
\draw (3.9,3.5) node {$H_{\alpha_2}(1)$};

\draw  (0.3,-4.4)--(0.3,4.4);
\draw (1.2,-4.4) node {$H_{\alpha_1+2\alpha_2}(c')$};

\draw (4.4,-4.2)--(-4.2,4.4);
\draw (4.2,-4.3) node {$H_{\alpha_1+\alpha_2}(c)$};
\end{tikzpicture}
\caption{  \textit{Examples of pole hyperplanes for Eisenstein series. }}
\end{figure}
\bigskip

The spectrum of $L^2(\Gamma\backslash G)$ arising from the Borel group  is described as follows.
For $\Lambda_0\in \re I$ the function
\begin{equation*}
 \tilde \phi(g)\:=\:\int_{\re\Lambda=\Lambda_0}E_B(g,\phi,\Lambda)~d\Lambda
\end{equation*}
belongs to $L^2(\Gamma\backslash G)$, and the scalar product of two such can be expressed by
\begin{equation}\label{allgemeine_skalarproduktformel}
 \langle \tilde\phi,\tilde \psi\rangle\:=\:
\int_{\re\Lambda=\Lambda_0} f(\Lambda)~d\Lambda\:,
\end{equation}
where
\begin{equation*}
 f(\Lambda)\:=\:  \sum_{s\in W} \left( M(s,\Lambda)\Phi(\Lambda),\Psi(-s\bar\Lambda)\right)\:
\end{equation*}
for certain functions $\Phi,\Psi:\mathfrak a_\CC^\ast \to V$ of fast decay.
We choose a path in $\mathfrak a^\ast$ from $\Lambda_0$ to zero which omits the intersections $H_\gamma(c)\cap H_{\gamma'}(c')$ of hyperplanes.
Evaluating the integral (\ref{allgemeine_skalarproduktformel}) by the residue theorem we get a term 
\begin{equation*}
 \int_{\re\Lambda=0} f(\Lambda)~d\Lambda\:,
\end{equation*}
which gives rise to a $2$-dimensional parametrized $(\re \Lambda=0)$ spectral component. (By the functional equations, it is enough to parametrize this component by a cone contained
in $\re \Lambda=0$. But for our purpose the location $\re \Lambda=0$ is sufficient.)
This cannot be decomposed into irreducibles and gives rise to the so-called $2$-dimensional 
spectral component.
Further we get residual terms at any intersection point $z=z(\gamma,c)$ of the path with a hyperplane $H_\gamma(c)$. For simple roots $\gamma=\alpha_j$ we have
$^\dagger\!z=\frac{c}{2}\alpha_j$ and the terms are (up to constants)
\begin{equation*}
\int_{^\bullet\Lambda=^\bullet\!z+i^\bullet\!\aaa^\ast} \mathop{res}\nolimits_{^\dagger\!\Lambda=^\dagger\!z}f(^\bullet\!\Lambda+\:^\dagger\!\Lambda)~d^\bullet\!\Lambda\:.
\end{equation*}
For these we again apply the residue theorem  to get terms 
\begin{equation*}
 \int_{^\bullet\Lambda=i^\bullet\!\aaa^\ast}\mathop{res}\nolimits_{^\dagger\!\Lambda=\frac{c}{2}\alpha_j}  f(^\bullet\!\Lambda+\:^\dagger\!\Lambda))~d^\bullet\!\Lambda\:,
\end{equation*}
which give rise to spectral components which cannot be decomposed into irreducibles and are parametrized by the $1$-dimensional sets
\begin{equation*}
 K_{\alpha_j}(c)\: =\:\frac{c}{2}\alpha_j+i(^\bullet\aaa)^\ast\:
\end{equation*}
for simple roots. 
For non-simple roots $\gamma$ we can express  $K_\gamma(c)=sK_{\alpha_j}(c)$ for $s\in W$ with $\gamma=s\alpha_j$.
Further we get residual terms 
\begin{equation*}
 \mathop{res}\nolimits_{^\dagger\!\Lambda=^\dagger\!\tilde\Lambda}\mathop{res}\nolimits_{^\dagger\!\Lambda=\frac{c}{2}\alpha_j} f(^\bullet\!\Lambda+^\dagger\!\Lambda)
\end{equation*}
at intersection points $\tilde\Lambda=H_{\alpha_j}(c)\cap H_{\gamma'}(c')$ of hyperplanes between $z=\frac{c}{2}\alpha_j+^\bullet\!z$ and $\frac{c}{2}\alpha_j$,
which give rise to spectral components which decompose to irreducible representations with infinitesimal character $\tilde\Lambda$.
(Correspondingly for non-simple roots.)
As there is the constraint $0\leq c\leq 1$ for the pole hyperplanes, the real parts of all these spectral components belong to a 
circle of radius $\lvert\!\lvert \delta\rvert\!\rvert=\sqrt{5}$.
These cases exhaust the spectrum $L^2(B)\subset  L^2(\Gamma\backslash G)$.

If there is more than one $\Gamma$-conjugation class of the Borel subgroup $B$, then any of these classes produces a picture
like this. But as all Borel subgroups are conjugated under $\Sp_2(\ZZ)$, their Langlands parameters  can be
mapped into $\aaa_\CC^\ast$ accordingly. They can produce additional hyperplanes $H_\gamma(c)$, but $c$ is still
bounded by one. 
\bigskip

There are  shares of the Klingen and Siegel parabolics $P=\:^\bullet\!P$ occuring analogously. We have Eisenstein series
\begin{equation*}
 E_P(g,\phi,\Lambda)\:=\:\sum_{\gamma\in(\Gamma\cap P)\backslash\Gamma} a(\gamma g)^{\delta_P+\Lambda}\phi(\gamma g,\Lambda)\:,
\end{equation*}
for $g=namk$, where $k=k(g)\in K$ and $a=a(g)\in\: ^\bullet\!A$, $n=n(g)\in\: ^\bullet\!N$ and $m=m(g)\in\: ^\bullet\!M$.
Here $\Lambda$ belongs to $^\bullet\!\aaa_\CC^\ast\cong \CC$ and 
 $\phi$ is such that for all $g\in G$ the function $\phi(gk^{-1})$ is of the same weight and  the function $\phi(mg)$ belongs to 
a simple admissible subspace of $L_0^2((\Gamma\cap \:^\bullet\!M)\backslash \:^\bullet\!M)$. The Eisenstein series are convergent for $\re\Lambda>>0$ and have meromorphic continuation
to $\CC$ with a finite number of poles $\Lambda=c$.
The functions
\begin{equation*}
 \tilde\phi(g)\:=\:\int_{\re\Lambda=\Lambda_0} E_P(g,\phi,\Lambda)~d\Lambda
\end{equation*}
belong to $L^2(\Gamma\backslash G)$ for $\Lambda_0>>0$ and are spectrally decomposed by a residue process like above. This produces  irreducible spectral components indexed by $\Lambda=c$ and
a continuous component indexed by $\re\Lambda=0$. Translated to the whole image in $\aaa_\CC^\ast$ by adding the character terms of induction to receive the infinitesimal characters, 
these also correspond to points in $\aaa_\CC^\ast$ or a $1$-dimensional spectral component contained
 in the $2$-dimensional parameter $\re \Lambda=0$ in $\aaa_\CC^\ast$.

The share of the parabolic group G itself is given by  discrete series representations.
So we have
\begin{equation}\label{eqn_spectral_decompo}
 L^2(\Gamma\backslash G)\:=\:  L^2_{cont}(\Gamma\backslash G)
\:\bigoplus_\Lambda\: L^2_\Lambda(\Gamma\backslash G)\:.
\end{equation}
Here $L^2_\Lambda(\Gamma\backslash G)$ denotes the dirct sum of the isotypical components of irreducible unitary representations 
with infinitesimal characters $\Lambda=(\Lambda_1,\Lambda_2)$.
And the continuous spectrum $L^2_{cont}(\Gamma\backslash G)$  decomposes into
\begin{equation}\label{eqn_cont_spectral_decompo}
 L^2_{cont}(\Gamma\backslash G) \:=\:  L^2_{\re\Lambda=0}(\Gamma\backslash G)
\:\bigoplus_{\gamma\in\Sigma^+\!,\:c} \:L^2_{\gamma,c}(\Gamma\backslash G) \:.
\end{equation}
We fix a description in coordinates.
\begin{lem}\label{lemma_1_dim_sectrum}
There are finitely many real values $0\leq c\leq 1$ such that the
$1$-dimensional components of the continuous spectrum in $L^2(B)$ are given in
Langlands coordinates $\Lambda\in\aaa_\CC^\ast$ by
\begin{eqnarray*}
 K_{\alpha_1}(c)&=& \frac{c}{2}(0,2)+i\RR(1,0)\:,\\
 K_{\alpha_2}(c)&=& \frac{c}{2}(1,-1)+i\RR(1,1)\:,\\
K_{\alpha_1 +\alpha_2}(c)&=& \frac{c}{2}(1,1)+i\RR(1,-1)\:,\\
K_{\alpha_1+2\alpha_2}(c)&=& \frac{c}{2}(2,0)+i\RR(0,1)\:.
\end{eqnarray*}
\end{lem}
By the functional equations of Eisenstein series, one can restrict these para\-me\-tri\-za\-tions to a half-line in $i\RR$. But as the above description is sufficiently exact for the following,
we stick to this simple notation.
\begin{prop}\label{prop_auftreten_kontinuierliches spektrum}
The $1$-dimensional continuous component $K_{\alpha_1}(1)=(i\RR,1)$ actually occurs in $L^2(\Gamma\backslash G)$.
\end{prop}
\begin{proof}
 In the Eisenstein series $E_B(g,\phi,\Lambda)$ for $B$ we choose $\phi\equiv 1$.
We use (\ref{E-reihen-zerlegung}) for the Klingen parabolic, i.e. the parabolic $^\bullet\!P$ with split component $^\bullet\!\aaa=\ker(\alpha_1)$ and
$^\bullet\!M\cong \SL_2(\RR)\times Z_2$. 
The inner Eisenstein series 
\begin{equation*}
 ^\bullet\!E(g,\phi,\:^\dagger\!\Lambda)\:=\:\sum_{\bar\gamma\in(\Gamma\cap \:^\dagger\!P)\backslash 
(\Gamma\cap\:^\bullet\!M)}\:
^\dagger a(\bar\gamma g)^{1+\:^\dagger\!\Lambda}\phi(\bar\gamma g) \:,
\end{equation*}
where $^\dagger\!\Lambda=\Lambda_2$, is a $\mathop{SL}_2(\RR)$-Eisenstein series which has got a simple pole in $^\dagger\!\Lambda=1$.
So the residue
\begin{equation*}
 \mathop{res}\nolimits_{^\dagger\!\Lambda=1}E_B(g,1,\Lambda)\:=\:
 \sum_{\gamma\in(\Gamma\cap\:^\bullet\!P)\backslash \Gamma}\:^\bullet\!a(\gamma g)^
{^\bullet\!\delta+\:^\bullet\!\Lambda}
\mathop{res}\nolimits_{^\dagger\!\Lambda=1}
\:^\bullet\!E(\gamma g,\phi,\:^\dagger\!\Lambda)
\end{equation*}
is nonzero, as well is the residue $\mathop{res}\nolimits_{^\dagger\!\Lambda=1}M(s_{\alpha_2},\Lambda)$ of the scattering operator. In consequence,
the $L^2(\Gamma\backslash G)$-function
\begin{equation*}
 \int_{\re\Lambda=\Lambda_0}E_B(g,1,\Lambda)~d\Lambda
\end{equation*}
has a non-zero spectral component at $K_{\alpha_1}(1)$. 
\end{proof}

\section{Resolvents}\label{sec_resolvents}
Let $\Sigma=\Sigma^+\cup(-\Sigma^+)$ be the complete system of roots.
For  $\alpha\in\Sigma$ let $\check\alpha$ denote its coroot. Let $u$ and $v$ be complex variables.
Define 
\begin{empheq}[box=\fbox]{align}\label{wahre_gestalt_von_Dplusminus}
    \begin{split}
      D_+(u,\Lambda) &=  \prod_{\alpha\in\Sigma\textrm{ long}} \bigl(\check{\alpha}(\Lambda)-u\bigr)\:,\\
      D_-(v,\Lambda)  &= \prod_{\alpha\in\Sigma\textrm{ short}} \bigl(\check{\alpha}(\Lambda)-v\bigr)\:.
    \end{split}
  \end{empheq} 
In terms of the chosen basis above, for $\Lambda=(\Lambda_1,\Lambda_2)\in\aaa_\CC^\ast$ we have
\begin{eqnarray*}
 D_+(u,\Lambda)&=& (\Lambda_1^2-u^2)(\Lambda_2^2-u^2)\:,\\
 D_-(v,\Lambda)&=& \bigl((\Lambda_1+\Lambda_2)^2-v^2\bigr)\bigl((\Lambda_1-\Lambda_2)^2-v^2\bigr)\:.
\end{eqnarray*}
By Cor.~\ref{harish-chandra}, $D_+(u,\Lambda)$, respectively $D_-(v,\Lambda)$, is the image of 
\begin{eqnarray}
 D_+(u)&:=& \frac{1}{2}\bigl(C_1^2-C_2+11C_1-2(u^2-1)C_1+2(u^2-1)(u^2-4)\bigr)\:,\label{Dplus_in_koordinaten}\\
D_-(v)&:=& 2C_2-C_1^2-34C_1-2(v^2-9)C_1+(v^2-9)(v^2-1)\:,\label{Dminus_in_koordinaten}
\end{eqnarray}
respectively, under the Harish-Chandra homomorphism.
In section~\ref{sec_casimir_action_on_P} these Casimir operators will turn out to be the right choice for applications to the 
Poincar\'e series. For this, we distinguish between the variables $u$ and $v$, 
although this is redundant for this paragraph.

Being Casimir operators, $D_+(u)$ and $D_-(v)$ respect the spectral decomposition of $L^2(\Gamma\backslash G)$
(see (\ref{eqn_spectral_decompo}), (\ref{eqn_cont_spectral_decompo}) of section~\ref{section_spectrum}).
 So $D_+(u,\Lambda)$ and $D_-(v,\Lambda)$ describe their action on the spectral component parametrized by $\Lambda$.
We use this to study the existence of their  resolvents $R_+(u)$ and $R_-(v)$
in the variable $u$, respectively $v$.
For the existence of the resolvent $R_+(u)$ in some point $u\in \CC$, we have to check that $D_+(u)^{-1}$ exists and is 
bounded (see~\cite{hirzebruch-scharlau}, for example).
 Whenever defined for all $u$ in an open set $U\subset\CC$, the resolvent $R_+(u)$
  is analytic on $U$. (Analogously for $R_-(v)$.)
\bigskip

{\it The discrete spectrum.}
Let $\Lambda=(\Lambda_1,\Lambda_2)$ be the  parameter of a fixed discrete constituent  
$L^2_\Lambda(\Gamma\backslash G)_\kappa$.
Then $D_+(u,\Lambda)$ is a polynomial of degree four in $u$. As such it has at most four zeros.
Apart from these zeros,
\begin{equation*}
 R_+(u)\::=\:\frac{1}{(\Lambda_1^2-u^2)(\Lambda_2^2-u^2)}
\end{equation*}
is a constant, thus bounded, operator. While at a zero of $D_+(u,\Lambda)$, the resolvent $R_+(u)$ obviously has a pole.
The resolvent $R_-(v)$ is dealed with completely analogously. 
As there are only finitely many irreducible constituents, we get:
\begin{prop}\label{prop_discrete_spec}
On the discrete spectrum $\oplus_\Lambda L^2_\Lambda(\Gamma\backslash G)$, the resolvents $R_+(u)$ and $R_-(v)$
exist as meromorphic functions in the complex variable $u$ (resp. $v$), each having only a finite number of poles.
\end{prop}

{\it The $2$-dimensional continuous spectrum.}
Next we look at the case of the $2$-dimensional continuous spectrum, which is para\-me\-trized by $\re \Lambda =0$ 
(including $1$-dimensional continuous components occuring in $L^2(P)$, where $P$ is a Siegel or Klingen parabolic).
So we may assume $\Lambda=(it_1,it_2)$ for real $t_1,t_2$. Then
\begin{equation*}
 D_+(u)\:=\:(t_1^2+u^2)(t_2^2+u^2)
\end{equation*}
is a polynomial in $t_1,t_2$ which 
never vanishes if we assume $\re u>0$. Especially, the minimum
\begin{equation*}
 c(u)\::=\:\min_{t_1,t_2\in\RR}\lvert D_+(u)\rvert\:>\:0
\end{equation*}
exists. Thus, the inverse $D_+(u)^{-1}$ is bounded by $c(u)^{-1}$, and the resolvent $R_+(u)$ exists on the $2$-dimensional 
continuous spectrum as long as $\re u>0$.
But for $\re u=0$ we get zeros $(t_1,t_2)$ of $D_+(u)$, which are very difficult to study and which may  
produce essential singularities of the resolvent.
Similarly, for $D_-(v)=\bigl((t_1+t_2)^2+v^2\bigr)\bigl((t_1-t_2)^2+v^2\bigr)$ the resolvent $R_-(v)$ exists as long as $\re v>0$.
To sum up:
\begin{prop}\label{prop_cont_spec_2-dim}
On the $2$-dimensional continuous spectrum $L^2_{\re\Lambda=0}(\Gamma\backslash G)$, 
the resolvent $R_+(u)$ (respectively $R_-(v)$) is an analytic function for $\re u>0$ (respectively $\re v>0$).
\end{prop}

{\it The $1$-dimensional continuous spectrum.}
For the remaining $1$-dimensional components of the continuous spectrum parametrized by $K_\gamma(c)$ 
(see Lemma~\ref{lemma_1_dim_sectrum}), we first determine the zeros of $D_+(u)$ (respectively $D_-(v)$) on $K_\gamma(c)$.
For a fixed $u$ (respectively $v$) we denote by
\begin{eqnarray*}
 N_+(u)&=&\{\Lambda\in\aaa_\CC^\ast\vert D_+(u,\Lambda)=0\}\:,\\
N_-(v)&=&\{\Lambda\in\aaa_\CC^\ast\vert D_-(v,\Lambda)=0\}\:
\end{eqnarray*}
the set of zeros of $D_+(u,\Lambda)$ (respectively $D_-(v,\Lambda)$) in $\aaa_\CC^\ast$ (see Figure~\ref{bild_nst_von_D+_und_D-}).
\begin{figure}
\begin{tikzpicture}
 \draw [->](0,-3) -- (0,5.4);
 \draw (-0.25,5.4) node {$\Lambda_2$};
 \draw [->,style=thick] (0,0)--(0,4);
 \draw (0.25,3.9) node {$\alpha_1$};
\draw [->](-3,0) -- (5.4,0);
\draw (5.4,-0.3) node {$\Lambda_1$};
\draw [->,style=thick] (0,0)--(2,-2);
\draw (2.25,-2) node {$\alpha_2$};

\draw (4,2)  circle (1pt);
\draw  (3.8,2.1) node {$\delta$};
\filldraw [lightgray] (5.4,2)--(4,2) -- (5.4,3.4);
\draw [dashed](-3,1) --(5.4,1);
\draw [dashed](-3,-1) -- (5.4,-1);
\draw (1.5,5) node {\small $N_+(c)$};
\draw [dashed](1,-3) -- (1,5.4);
\draw [dashed](-1,-3) -- (-1,5.4);
\draw [dotted](-3,-1.2)--(3.6,5.4);
\draw [dotted](-1.2,-3)--(5.4,3.6);
\draw [dotted](-3,1.2)--(1.2,-3);
\draw (-2.9,5.0) node {$N_-(c')$};
\draw [dotted](4.8,-3)--(-3,4.8);

\draw (0,1) node {$\times$};
\draw (-0.55,0.7) node {\small $K_{\alpha_1}(c)$};
\draw (0,1.8) node {$\times$};
\draw (-0.55,2) node {\small $K_{\alpha_1}(c')$};
\end{tikzpicture}
\caption{  \textit{Examples for the zeros  $N_+(u)$ and $N_-(v)$ within $\mathfrak a^\ast$, 
and their intersections with parts of the $1$-dimensional continuous spectrum (real image).}}
\label{bild_nst_von_D+_und_D-}
\end{figure}
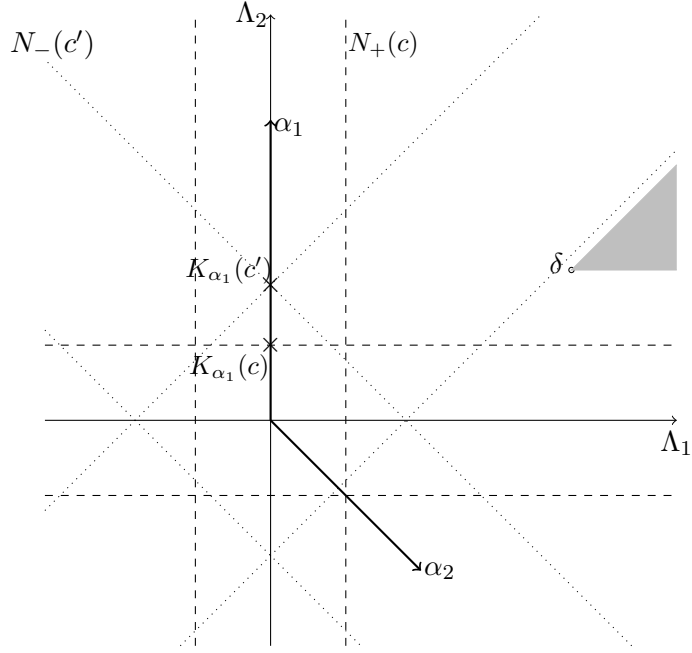
Having this restriction by Proposition~\ref{prop_cont_spec_2-dim} nevertheless, we restrict to   $\re u>0$, $\re v>0$. 
We easily find:
\begin{eqnarray*}
 K_{\alpha_1}(c)\cap N_+(u)&=&\left\{\begin{array}{ll}K_{\alpha_1}(c),&\textrm{ if } u= c\\
                                \emptyset, &\textrm{ else }
                               \end{array}
\right.\\
 K_{\alpha_2}(c)\cap N_+(u)&=&\left\{\begin{array}{ll}
\{\frac{c}{2}(1,-1)\pm iy(1,1)\},&\textrm{ if } u= \frac{c}{2}+iy\\
                                \emptyset, &\textrm{ else }
                               \end{array}
\right.\\
 K_{\alpha_1+\alpha_2}(c)\cap N_+(u)&=&\left\{\begin{array}{ll}
\{\frac{c}{2}(1,1)\pm iy(1,-1)\},&\textrm{ if } u= \frac{c}{2}+iy\\
                                \emptyset, &\textrm{ else }
                               \end{array}
\right.\\
 K_{\alpha_1+2\alpha_2}(c)\cap N_+(u)&=&\left\{\begin{array}{ll}K_{\alpha_1+2\alpha_2}(c),&\textrm{ if } u= c\\
                                \emptyset, &\textrm{ else }
                               \end{array}
\right.
\end{eqnarray*}
\begin{eqnarray*}
 K_{\alpha_1}(c)\cap N_-(v)&=&\left\{\begin{array}{ll}
\{\frac{c}{2}(0,2)\pm iy(1,0)\},&\textrm{ if } v= c+iy\\
                                \emptyset, &\textrm{ else }
                               \end{array}
\right.\\
 K_{\alpha_2}(c)\cap N_-(v)&=&\left\{\begin{array}{ll}
K_{\alpha_2}(c),&\textrm{ if } v= c\\
                                \emptyset, &\textrm{ else }
                               \end{array}
\right.\\
 K_{\alpha_1+\alpha_2}(c)\cap N_-(v)&=&\left\{\begin{array}{ll}
K_{\alpha_1+\alpha_2}(c),&\textrm{ if } v= c\\
                                \emptyset, &\textrm{ else }
                               \end{array}
\right.\\
 K_{\alpha_1+2\alpha_2}(c)\cap N_-(v)&=&\left\{\begin{array}{ll}
\{\frac{c}{2}(2,0)\pm iy(0,1)\},&\textrm{ if } v= c+iy\\
                                \emptyset, &\textrm{ else }
                               \end{array}
\right.
\end{eqnarray*}
We discuss  the  component $L^2_{\alpha,c}(\Gamma\backslash G)_\kappa$ parametrized 
by the $1$-dimensional set
\begin{equation*}
 K_{\alpha_1}(c)\:=\:(i\RR,c)
 \end{equation*}
in detail. For this we may assume $\Lambda=(it,c)$. Then,
\begin{equation*}
 D_+(u,\Lambda)\:=\:-(c^2-u^2)(t^2+u^2)
\end{equation*}
is zero for $\re u>0$ if and only if $u= c$.
As the polynomial $t^2+u^2$ does not vanish for any $u$ with $\re u>0$, the minimum
\begin{equation*}
 m_+(u)\::=\:\min_{t\in\RR}\lvert t^2+u^2\rvert \:>\:0
\end{equation*}
exists. Thus, 
\begin{equation*}
 \lvert D_+(u,\Lambda)\rvert \:\geq\:\lvert c^2-u^2\rvert\cdot m_+(u)
\end{equation*}
is bounded away from zero if $u\not=c$.
But $(c^2-u^2)R_+(u)$ is an operator on this component which is bounded by $m_+(u)^{-1}$, and which differs from $R_+(u)$ by
the constant $(c^2-u^2)$, 
so the resolvent $R_+(u)$ itself just has a pole of order one in $u=c$.
Indeed, applying $R_+(u)$ to $\phi$ in the scalar product fomular (\ref{allgemeine_skalarproduktformel})
we get a share in the $K_{\alpha_1}(c)=(i\RR,c)$-component of
\begin{equation*}
 \frac{1}{u^2-c^2}
  \int_{\RR}\frac{1}{t^2+u^2}\mathop{res}\nolimits_{\Lambda_2=\frac{c}{2}}  f(it+\Lambda_2)~dt\:,
\end{equation*}
which has a simple pole in $c=u$ and is bounded else.

In contrast,  the  operator 
\begin{equation*}
 D_-(v,(it,c))\:=\:\bigl((c+it)^2-v^2\bigr)\bigl((c-it)^2-v^2\bigr)
\end{equation*}
has the zeros  $v\in c+i\RR$ in case $\re v>0$.
Within $\CC$, this set is a real line which we cannot bypass. The resolvent $R_-(v)$ cannot be established 
 as a meromorphic function for 
$\re v\leq c$. But for $\re v >c$ the minimum
\begin{equation*}
 m_-(v)\::=\:\min_{t\in\RR}\left\vert \bigl((c+it)^2-v^2\bigr)\bigl((c-it)^2-v^2\bigr)\right\vert \:>\:0
\end{equation*}
exists. So does the resolvent, being bounded by $m_-(v)^{-1}$.

One easily checks that these two cases essentially describe the other $1$-di\-men\-sio\-nal components:
If the intersection $K_\gamma(c)\cap N_+(u)$ (resp. $K_\gamma(c)\cap N_-(v)$) is all of $K_\gamma(c)$, then the resolvent
$R_+(u)$ (resp. $R_-(v)$) is meromorphic on $\re u>0$ (resp. $\re v>0$) with a single pole of order one in 
$u=c$ (resp. $v=c$).
But if the intersection $K_\gamma(c)\cap N_+(u)$ (resp. $K_\gamma(c)\cap N_-(v)$) consists of at most two points,
then on the line $\re u=\frac{c}{2}$ (resp. $\re v=c$)  the resolvent
$R_+(u)$ (resp. $R_-(v)$) has singularities we cannot deal  with.

\begin{prop}\label{prop_cont_spec_1-dim}
On  the $1$-dimensional continuous spectral components of $L^2(B)$, the resolvent $R_+(u)$ exists for $\re u>\frac{1}{2}$
as a meromorphic function having  poles of order one for finitely many values of $u$, where $u=1$ occurs. 
The resolvent $R_-(v)$ exists  for $\re v>1$ as a holomorphic function.
\end{prop}
\begin{proof}[Proof of Proposition~\ref{prop_cont_spec_1-dim}]
By the arguments above, we get a bound $b_+$ for  $R_+(u)$ to exist,
\begin{equation*}
 b_+:=\max\{\frac{c}{2}, \:c \textrm{ occuring in one of the } K_{\alpha_2}(c), K_{\alpha_1+\alpha_2}(c)\}\:.
\end{equation*}
By Lemma~\ref{lemma_1_dim_sectrum}, we have $c \leq 1$, where $c=1$ actually occurs, as $H_{\alpha_1}(1)$ is one of
the two walls margining the cone of convergence of the Eisenstein series. So $b_+=\frac{1}{2}$. 
For $\re u>\frac{1}{2}$, the resolvent $R_+(u)$ has simple poles at $u=c$, for $c$ occurring in one of the 
$K_{\alpha_1}(c)$ and $K_{\alpha_1+2\alpha_2}(c)$. 
Similarly, let
\begin{equation*}
 b_-:=\max\{c,\:c \textrm{ occuring in one of the } K_{\alpha_1}(c), K_{\alpha_1+2\alpha_2}(c)\}\:\leq\:1\:. 
\end{equation*}
Then $R_-(v)$ is holomorphic for $\re v>b_-$.
\end{proof}


\section{Poincar\'e series}\label{section_konvergenz}
\begin{defn}\label{def_allgemeine_poincarereihe}
Let the genus $m$ be arbitrary.
Let  $s_1,s_2$ be complex variables, and let $\tau$ be a positive definite $(m,m)$-matrix with half-integral entries.
Define the  Poincar\'e series 
\begin{equation*}
 \mathcal P_\tau(g,s_1,s_2) :=\sum_{\gamma\in \Gamma_\infty\backslash \Gamma} H_\tau(\gamma g,s_1,s_2)\:\:,
\end{equation*}
where
\begin{equation*}
 H_\tau(g,s_1,s_2)\::=\: \frac{\exp(2\pi i\trace(\tau z))}{j(g,i)^{\kappa}}\trace(\tau y)^{s_1}
\det (y)^{s_2}\:.
\end{equation*}
\end{defn}
Due to their exponential term we call them Poincar\'e series  of exponential type. 
They  are of weight $\kappa$. 
The function $H_\tau(g,s_1,s_2)$ is
nonholomorphic (in the variable $g$) apart from $(s_1,s_2)=(0,0)$.
Klingen~\cite{Klingen} studied very closely defined series  $P_\tau^\kappa(g)=\sum h(\gamma g)$ for
$h(g)=\exp(2\pi i\trace(\tau z))\det (y)^{\frac{\kappa}{2}}$, which  converge for $\kappa>2m$.
But in section~\ref{sec_casimir_action_on_P_geschlecht_2}   we will fix  genus $m=2$ and $\kappa=2m=4$.
So the Poincar\'e series do not converge in their point of holomorphicity.
We will use other coordinates in this case which fit better into spectral theory:
\begin{defn}\label{definition_poincare} For $m=2$ define the Poincar\'e series $P_\tau(g,u,v)$ by
\begin{equation*}
 P_\tau(g,u,v)\::=\:\mathcal P_\tau(g,s_1,s_2),
\end{equation*}
where
\begin{equation*}
 s_1\:=\:\frac{v-2u-1}{2}\quad\textrm{ and }\quad s_2\:=\:\frac{u-(\kappa-m)}{2}\:.
\end{equation*}
By abuse of notation, we usually omit the dependence on $\tau$ in our notations, i.e.
$P(g,u,v):=P_\tau(g,u,v)$.
\end{defn}
\begin{thm}\label{Konvergenzbereich} 
The series
\begin{equation*}
S_\tau(z,l_1,l_2)\::=\: \sum_{\gamma\in\Gamma_\infty\backslash \Gamma}\exp(2\pi i\trace(\tau \gamma\Hop z))
\trace(\tau\im \gamma\Hop z)^{l_1}\det(\im\gamma\Hop z)^{l_2}
\end{equation*}
converges absolutely and uniformly on compact sets in the  cone given by:
\begin{enumerate}
 \item[(a)]
$\re l_1\geq 0\:$ and  $\:\re l_2>m$
\item[(b)]
$\re l_1\leq 0\:$ and  $\:\re (l_2+ \frac{l_1}{m})>m$.
\end{enumerate}
For $(l_1,l_2)$ fixed, it is absolutely bounded by a constant independent of $\tau$ and belongs to $L^2(\Gamma\backslash G)$.
\end{thm}
From this we deduce the properties of the Poincar\'e series. 
\begin{cor}\label{cor_konvergenzbereich_u_v}
 The Poincar\'e series $\mathcal P_\tau(g,s_1,s_2)$ of  weight $\kappa$  
converges absolutely and  uniformly on compact sets within
\begin{equation*}
 \{(s_1,s_2)\in\CC^2\mid \re(2s_2+\kappa)>2m\textrm{ and } \re(\frac{2}{m}s_1+2s_2+\kappa)>2m\}\:.
\end{equation*}
%
It belongs to $L^2(\Gamma\backslash G)$.
For $m=2$ the Poincar\'e series $P(g,u,v)$ have got
the cone of convergence (see Figure~\ref{bild_konvergenzbereich}) 
\begin{equation*}
A\:=\: \{(u,v)\in\CC^2\mid \re u>2 \textrm{ and } \re v>5\}\:.
\end{equation*}
\end{cor}
\begin{figure}
\begin{tikzpicture}
\pgftransformtriangle
{\pgfpoint{0pt}{0pt}}
{\pgfpoint{0.433pt}{-0.25pt}} 
{\pgfpoint{0pt}{0.5pt}}
\draw [->](-0.3,0) -- (7,0);
\draw (7,-0.2) node {\small $\re u$};
\draw [->](0,-0.3) -- (0,7);
\draw (-0.8,6.9) node {\small $\re v$};
\filldraw[lightgray] (1,7.5) -- (1,2.5) -- (7.5,2.5)--(7.5,10.75);
\draw (4,5.5) node {$A$};
\draw (1,-0.3)--(1,7.5);
\draw (2.6,1) node {\small $\re u=2$};
\draw (7.5,2.5) -- (-0.3,2.5);
\draw (-1.8,1.95) node {\small $\re v=5$};
\draw [dashed] (-0.3,-0.1) --(5.375,10.75);
\draw (8,11.5) node {\small $\re (v-2u)=0$};
\end{tikzpicture}
\caption{  \textit{Cone $A$ of convergence of the Poincar\'e series in case of genus $m=2$.}}
\label{bild_konvergenzbereich}
\end{figure}
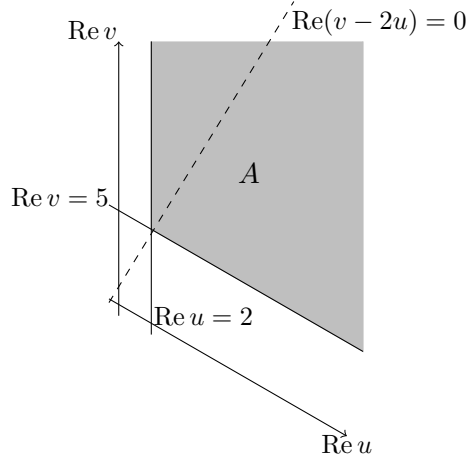
\begin{proof}[Proof of Corollary~\ref{cor_konvergenzbereich_u_v}]
The series in Theorem~\ref{Konvergenzbereich} and the Poincar\'e series have the same series of absolute values
if we identify
\begin{equation*}
 l_1\:=\:s_1\quad\quad\textrm{ and }\quad\quad l_2\:=\:s_2+\frac{\kappa}{2}\:,
\end{equation*}
as $j(g,i)^{-\kappa}=\det(y)^{\frac{\kappa}{2}}$ for $y=\im (g\cdot i)$.
%
The two domains of (a) and (b) glue together to the claimed one. 
The result for $P(g,u,v)$ is  a coordinate transform of this.
\end{proof}
To prove Theorem~\ref{Konvergenzbereich}, one can show absolute convergence  by estimating the 
 series $S_\tau(z,l_1,l_2)$ against Eisenstein series.
Then the two parts of the convergence area are quite obvious: While for $\re l_1\geq 0$, the factor $\trace(\tau y)^{l_1}$ 
is seized by
the exponential factor,  one uses Lemma~\ref{lemma_spuren_symmetrischer_matrizen} in case $\re l_1<0$ to seize the trace 
 by the determinant.
Then any of Borel, Klingen or Siegel Eisenstein series produce the same  constraints.
But for square integrability, we had to use truncations of   Eisenstein series, which is hard work for higher genus.
We prefer another approach inspired by~\cite{Klingen}. There, Klingen uses a ``reproducing integral'' 
on the Siegel halfplane $\mathcal H$ involving the function 
\begin{equation*}
 \frac{\det(\im z)\det(\im w)}{\lvert\det(z-\bar w)\rvert^2}\:.
\end{equation*}
Looked at from group level, this formula is nothing else than the leftinvariance of the Haar measure on $G/K$ applied 
to the $K$-biinvariant function above. The following observations
generalize Klingen's trick from this point of view. They will prove Theorem~\ref{Konvergenzbereich} as a special case.


Let $g_z$ respectively $g_w$ be elements of $G$ with $g_z\Hop i=z$ and $g_w\Hop i=w$ in $\mathcal H$ respectively.
We choose the Haar measure on $G/K$ to equal the measure on $\mathcal H$.
\begin{prop}\label{prop_klingen_trick_allg}

Let $f$ be some positive $K$-invariant function on $G$, and let
\begin{equation*}
 \mathcal P(g_z,f)\::=\:\sum_{\gamma\in \Gamma_\infty\backslash \Gamma}e^{M i\trace(\tau \gamma\cdot z)}f(\gamma g_z)\:,
\end{equation*}
where $M>0$ is some real constant.
Assume there is a $K$-biinvariant positive function $u$ on $G$ satisfying the following two conditions:
\begin{enumerate}
 \item [(i)]\label{bedingung_1}
$\int_{G/K}u(g_z^{-1}g_w)~dv_w=c_1<\infty$
\item[(ii)]\label{bedingung_2}
For
\begin{equation*}
 L(g_z):=\sum_{\alpha\in\Gamma_\infty}\int_{\alpha.\mathcal F}u(g_z^{-1}g_w)\frac{e^{M\trace(\tau \im z)}}{f(g_z)}~dv_w
\end{equation*}
there exists a constant $c_2(z)>0$ such that $L(\gamma g_z)\geq c_2(z)$ for all $\gamma\in\Gamma$.
\end{enumerate}
Then $\mathcal P(g_z,f)$ is absolutely convergent with $\lvert \mathcal P(g_z,f)\rvert\leq \frac{2c_1}{c_2(z)}$.
\end{prop}
 Notice that the constant $c_1$ of condition (i) is indeed independent of $g_z$, as the Haar integral  is leftinvariant.
While the constant $c_2(z)$ of condition (ii) only depends on  $\Gamma g_zK$. For example, we may choose the representative 
 $z\in\H$ to be reduced modulo $1$.
%
\begin{proof}[Proof of Proposition~\ref{prop_klingen_trick_allg}]
 We have
\begin{eqnarray*}
 \sum_{\gamma\in\Gamma_\infty\backslash \Gamma}\lvert L(\gamma g_z)e^{-M \trace(\tau \gamma\cdot z)}f(\gamma g_z)\rvert
&=&\sum_{\gamma\in\Gamma_\infty\backslash \Gamma}\sum_{\alpha\in\Gamma_\infty}
\int_{\alpha.\mathcal F}\!\!\!u(g_z^{-1}\gamma^{-1}g_w)~dv_w\\
&=&\sum_{\gamma\in\Gamma}\int_{\gamma.\mathcal F}\!\!\!u(g_z^{-1}\gamma g_w)~dv_w=2c_1,
\end{eqnarray*}
by condition (i). As $L(\gamma g_z)\geq c_2$ by condition (ii), we deduce from the last equality
\begin{equation*}
 \lvert \mathcal P(g_z,f)\rvert\leq\sum_{\gamma\in\Gamma_\infty\backslash \Gamma}
e^{-M \trace(\tau \im(\gamma\cdot z))}f(\gamma g_z)\leq\frac{2c_1}{c_2(z)}\:.\qedhere
\end{equation*}
\end{proof}
The function
\begin{equation*}
 u(g_z,g_w)\:=\: \frac{\det(\im z)\det(\im w)}{\lvert\det(z-\bar w)\rvert^2}
\end{equation*}
is  constant on the diagonal, symmetric, positive, and leftinvariant, 
i.e.  satisfies $u_j(hg_z,hg_w)=u_j(g_z,g_w)$ for all $h\in G$. 
\begin{lem}\label{lemma_spuren_symmetrischer_matrizen}
 Let $S\in M_{m,m}(\RR)$ be a symmetric positive definite matrix. 
\begin{itemize}
 \item [(a)]
There is a constant $k=k(m)>0$ such that
\begin{equation*}
 \trace(S)^m\:>\:k\cdot \det(S)\:.
\end{equation*}
\item[(b)]
There are positive constants $m_1=m_1(S)$ and $m_2=m_2(S)$ such that for any symmetric positive definite $Y\in M_{m,m}(\RR)$ 
\begin{equation*}
 m_1\trace(Y)\:\leq \:\trace(SY)\:\leq\:m_2\trace(Y)\:.
\end{equation*}
\end{itemize}
\end{lem}
\begin{proof}This is wellknown and easy to prove.
Choosing $k=\frac{1}{2}$ and $m_1=S=m_2$, the Lemma is obviously true in case $m=1$.
So let $m\geq 2$.
 We may assume $S=\mathop{diag}(s_1,\dots,s_m)$ to be diagonal, as for any $S>0$ there exists an orthogonal matrix 
$L$ such that $L'SL$ is  diagonal.
Then use $\trace(S)=\trace(L'SL)$ in (a), while  (b) being true for all $Y$ is equivalent to (b) being true for all $LYL'$.
For (a)  we get by induction on $m$
\begin{equation*}
 \trace(S)^m\:=\:(s_1+\dots+s_m)^m\:=\:\trace(S^m)+m! \cdot\det(S)+R\:,
\end{equation*}
where $R$ is a polynomial  in $s_1,\dots,s_m$ with positive coefficients, and 
where in any monomial at least two but at most $m-1$ different $s_j$ occur.
So
\begin{equation*}
 0\:<\:\trace(S^m)+R=\trace(S)^m-m! \cdot\det(S)\:,
\end{equation*}
which implies the claim with $k=m!$.
For (b) we have $\trace(SY)=\sum_{j=1}^m s_jY_{jj}$. Choosing $m_1=\min_j\{s_j\}$ and $m_2=\trace(S)$, the estimations
hold.
\end{proof}
\begin{lem}\label{prop_sphaerisches_integral}
The integral
\begin{equation*}
  \int_{G/K} u(g_z,g_w)^{k}~dg_w
\end{equation*}
exists if and only if $k>m$.
\end{lem}
\begin{proof}[Proof of Lemma~\ref{prop_sphaerisches_integral}.]
In view of the left-invariance of the Haar measure, we may assume $z=i$.
Let $w=x+iy$.
 We use the Cayley transform of the Siegel halfplane to the unit circle 
$\mathbb E_m=\{\zeta=\zeta'\in M_{m,m}(\mathbb C)\mid 1-\zeta\bar\zeta>0\}$ given by
\begin{equation*}
 w\:\mapsto\: \zeta=(w-i)(w+i)^{-1},
\end{equation*}
which implies $y=\frac{1}{4}(w+i)(1-\zeta\bar\zeta)\overline{(w+i)}$. We get
\begin{equation*}
\det(1-\zeta\bar\zeta)\:=\:4^m\cdot u(g_i,g_w) \:. 
\end{equation*}
Accordingly, 
\begin{equation*}
 \int_{G/K}u(g_i,g_w)^{k}~dg_w\:=\:
4^{-mk}\int_{\mathbb E_m}\det(1-\zeta\bar\zeta)^{k}~dv_\zeta\:.
\end{equation*}
Here the measure is chosen such that $dv_\zeta=\det(1-\zeta\bar\zeta)^{-(m+1)}d\xi~\d\eta$ for $\zeta=\xi+i\eta$.
By \cite[Theorem 2.3.1]{hua}, the integral $\int_{\mathbb E_m}\det(1-\zeta\bar\zeta)^{k}~dv_\zeta$ exists if and only if $k>m$. 
\end{proof}
While Lemma~\ref{prop_sphaerisches_integral}  is a tool for condition~(i) of Proposition~\ref{prop_klingen_trick_allg},
the following  is helpful for condition~(ii).
\begin{lem}\label{lemma_u_abschaetzungen}
There exists a compact set $C\subset\mathcal F$ with $\vol(C)>0$ and a polynomial $Q$
with positive coefficients
 such that for all 
$w\in C$ and all $z\in\mathcal H$ reduced modulo $1$
\begin{equation}\label{abschaetzung_u_2}
Q(\trace(\tau y))^{-1}\:\leq\: \frac{u_2(g_z,g_w)}{\det(y)}\:\leq\:  \textrm{const.}\:.
\end{equation}
\end{lem}
\begin{proof}[Proof of Lemma~\ref{lemma_u_abschaetzungen}] 
There is a compact set $C\subset \mathcal F$ with $\vol(C)>0$ such  that for all $w\in C$
and all $z$ reduced modulo $1$ we have
\begin{equation*}
 \lvert\det(z-\bar w)\rvert^2\:\leq P(\trace(\tau y))\:,
\end{equation*}
for some polynomial $P$ with strictly positive  coefficients.
Thus,
\begin{equation*}
Q(\trace(\tau y))^{-1}\leq \frac{u_2(g_z,g_w)}{\det (y)}\:,
\end{equation*}
where $Q$ equals $P$ up to a multiple constant depending on $C$. As this fraction is  bounded for $w$ and $z$
as above, 
(\ref{abschaetzung_u_2})  follows.
\end{proof}

\begin{proof}[Proof of Theorem~\ref{Konvergenzbereich}.]
Define for $z=x+iy$
\begin{equation*}
 h_\tau(z,l_1,l_2)\::=\:\exp(2\pi i\trace(\tau z))\trace(\tau z)^{l_1}\det(z)^{l_2}\:.
\end{equation*}
Then the series in question  is
\begin{equation*}
 S_\tau(z,l_1,l_2)\:=\:\sum_{\gamma\in\Gamma_\infty\backslash\Gamma}h_\tau(\gamma\Hop z,l_1,l_2) \:,
\end{equation*}
and we have
\begin{equation*}
 \lvert h_\tau(g,l_1,l_2)\rvert \:=\:e^{-2\pi\trace(\tau y)}\trace(\tau y)^{\re l_1}\det(y)^{\re l_2}\:.
\end{equation*}
So we will assume $l_1$ and $l_2$ to be real.
We prove part (a) first, so let $l_1\geq 0$.
We check conditions (i) and (ii) of Proposition~\ref{prop_klingen_trick_allg} for $u(g_z,g_w)^{l_2}$.
By Lemma~\ref{prop_sphaerisches_integral}, the constant
\begin{equation*}
 c_1(l_2)\::=\:\int_{G/K} u(g_z,g_w)^{l_2}~dw \:>\:0
\end{equation*}
exists iff $l_2>m$. So (i) is satisfied.
By Lemma~\ref{lemma_u_abschaetzungen}, we have for $z$ reduced modulo $1$ and $w\in C\subset \mathcal F$
\begin{equation*}
 \frac{u(g_z,g_w)^{l_2}}{\trace(\tau y)^{l_1}\det(y)^{l_2}}\: \geq \tilde Q(\trace(\tau y))^{-1}\:,
\end{equation*}
with $\tilde Q(X)=X^{l_1}Q(X)$. Thus,
\begin{eqnarray*}
 L(g_z)&:=&
\sum_{\alpha\in\Gamma_\infty}\:\int\limits_{\alpha.\mathcal F}^{}\!\!\!u(g_z,g_w)^{l_2}
\frac{e^{2\pi\trace(\tau y)}}{\trace(\tau y)^{l_1}\det(y)^{l_2}}~dv_w\\
&\geq& \int\limits_C\!\frac{e^{2\pi\trace(\tau y)}}{\tilde Q(\trace(\tau y))}~dv_w
\:\geq\: c_2(l_1,l_2)\:>0\:,
\end{eqnarray*}
where $c_2(l_1,l_2) =\vol(C)\cdot c_3$ for any bound $c_3>0$ satisfying
\begin{equation*}
 \frac{e^{2\pi X}}{\tilde Q(X)}>c_3\:.
\end{equation*}
so $c_2(l_1,l_2)$ is indeed independent of $z$. Especially, condition (ii) of Proposition~\ref{prop_klingen_trick_allg} is
satisfied and we deduce that $S_\tau(z,l_1,l_2)$ is absolutely convergent and bounded by 
\begin{equation*}
 \lvert S_\tau(z,l_1,l_2)\rvert \:\leq    \: \frac{c_2(l_1,l_2)}{c_1(l_2)}\:.
\end{equation*}
For $l_1$ and $l_2$ in any compact set we can choose the constants $c_1,c_2$ uniformly, so
the series converges absolutely and uniformly there. 
As the constants $c_1(l_2)$ and $c_2(l_1,l_2)$ do not depend on $g_z$, the series is square integrable,
 thus  belongs to $L^2(\Gamma\backslash G)$.
For part (b), that is $l_1\leq 0$, we use Lemma~\ref{lemma_spuren_symmetrischer_matrizen}(a) to estimate
\begin{equation*}
 \trace(\tau y)^{l_1}\:\leq\: k\cdot \det(y)^{\frac{l_1}{m}}
\end{equation*}
for a constant $k=k(m,\tau)>0$. Thus,
\begin{equation*}
 \lvert h_\tau(z,l_1,l_2)\rvert\:\leq\:k\cdot \lvert h_\tau(z,0,l_2+\frac{l_1}{m})\rvert
\end{equation*}
and accordingly, $S_\tau(z,l_1,l_2)$ is  majorized by $S_\tau(z,0,l_2+\frac{l_1}{m})$.
The latter is already seen to be absolutely convergent and to 
belong to $L^2(\Gamma\backslash G)$ in case $\re(l_2+\frac{l_1}{m})>m$.
\end{proof}
\section{Action of Casmir elements on Poincar\'e series}\label{sec_casimir_action_on_P}
We give formulae for the action of the Casimir elements $C_1$ and $C_2$ on the Poincar\'e series.
Some preparatory remarks are in due.
\begin{lem}\label{lemma_abschaetzung_summand}
Let $X\in\mathfrak U(\mathfrak g_\CC)$ be of degree $n$, then
\begin{equation*}
 \lvert XH_\tau(g,s_1,s_2)\rvert\:\leq\:c\cdot\sum_{j=0}^n\lvert H_\tau(g,s_1+j,s_2)\rvert
\end{equation*}
with a constant $c>0$ depending only on $X$ and $\tau$.
\end{lem}
\begin{cor}\label{cor_differentiation_poincare_reihen}
 Within the domain of convergence of the Poincar\'e series,
\begin{equation*}
 X\mathcal P_\tau(g,s_1,s_2)\:=\:\sum_{\gamma\in\Gamma_\infty\backslash \Gamma}X H_\tau(\gamma g,s_1,s_2)
\:\in\:L^2(\Gamma\backslash G)\:,
\end{equation*}
for any $X\in \mathfrak U(\mathfrak g_\CC)$.
\end{cor}
\begin{proof}[Proof of Corollary~\ref{cor_differentiation_poincare_reihen}]
By Lemma~\ref{lemma_abschaetzung_summand},
the right hand side is majorized by a sum of convergent Poincar\'e series $\mathcal P_\tau(g,s_1+j,s_2)$. 
Thus, it belongs to $L^2(\Gamma\backslash G)$ and equals the left hand side.
\end{proof}
We first remark the following basic Lemma.
\begin{lem}\label{Eplusminus_auf_j}
Let $g=p\tilde k$, where $\tilde k\in K$ corresponds to $k\in U$ and let $J=J(g)=T'^{-1}k$. Then $j(g,i)=\det (J(g))$. 
Fir the basis elements of $\mathfrak G_\CC$ we have
\begin{equation*}
 \begin{array}{ccc}
  (E_-)_{ab} J=0, & (E_+)_{ab}J=-2\bar JX^{(ab)}, & B_{ab}J=-Je_{ba},\\
 E_-j(g,i)=0, &  E_+j(g,i)= -2j(g,i)\bar k'\bar k,& B_{ab}j(g,i)=-j(g,i)\delta_{ab},\\
(E_-)_{ab}\bar J=2 JX^{(ab)},&(E_+)_{ab} \bar J=0,& B_{ab}\bar J=\bar J e_{ab},\\
E_-\overline{j(g,i)}=2\overline{j(g,i)} k' k, &  E_+j(g,i)=0 ,& B_{ab}\overline{j(g,i)}=\overline{j(g,i)}\delta_{ab}.
 \end{array}
\end{equation*}
For functions $h(z)$ of the Siegel halfspace, we have $Bh(z)=0$ and (\cite{weissauersLN}, Chapter~3)
\begin{equation}\label{EminusaufH}
 E_-h(z)=(-4i)J'\left(y\bar\partial(h(z))y\right)J,\quad E_+h(z)=4i\bar J'\left(y\partial(h(z))y\right)\bar J\:.
\end{equation}
\end{lem}
Here $\partial=(\partial_{ab})_{ab}=\frac{1+\delta_{ab}}{2}\frac{1}{2}(\partial_{x_{ab}}-i\partial_{y_{ab}})$.
For convenience, we collect some easy formulae:
\begin{eqnarray*}
\bar \partial(e^{2\pi i\trace(\tau z)})=0\: , 
&& \partial(e^{2\pi i\trace(\tau z)})=2\pi i\tau e^{2\pi i\trace(\tau z)}\:,\\
\partial_y(\trace(\tau y))=\tau, && \partial_y(\det( y))=\det( y)y^{-1}\:,\\
 \partial_y(y_{ab})=X^{(a,b)}\:, &&\bar\partial (f(y))=\frac{i}{2}\partial_y(f(y))=-\partial(f(y))\:.
\end{eqnarray*} 
\begin{proof}[Proof of Lemma~\ref{lemma_abschaetzung_summand}]
The function 
\begin{equation*}
 H_\tau(g,s_1,s_2)\:=\: \frac{\exp(2\pi i\trace(\tau z))}{j(g,i)^{\kappa}}\trace(\tau y)^{s_1}\det (y)^{s_2}
\end{equation*}
belongs to $\mathcal C^\infty$. Using Lemma~\ref{Eplusminus_auf_j}, we compute the action of the basis of $\mathfrak g_\CC$:
\begin{eqnarray*}
 B_{ab}H_\tau(g,s_1,s_2)&=&\delta_{ab}\kappa H_\tau(g,s_1,s_2),\\
(E_-)_{ab} H_\tau(g,s_1,s_2)&=&2s_1H_\tau(g,s_1-1,s_2)(k'T'\tau Tk)_{ab}\\
&&+2s_2 H_\tau(g,s_1,s_2)(k'k)_{ab}\:,\\
(E_+)_{ab} H_\tau(g,s_1,s_2)&=& 2(\kappa+s_2) H_\tau(g,s_1,s_2)(\bar k'\bar k)_{ab}\\
&&-8\pi H_\tau(g,s_1,s_2)(\bar k'T'\tau T\bar k)_{ab}\\
&&+2s_1 H_\tau(g,s_1-1,s_2)(\bar k'T'\tau T\bar k)_{ab}\:.
\end{eqnarray*}
Notice that any single element of $T'\tau T$ is seized by $c(\tau)\trace(\tau y)$, for some constant $c(\tau)>0$.
So, as $K$ is compact, the  claim follows for elements $X$ of $\mathfrak g_\CC$.
The Lemma  follows for arbitrary $X\in\mathfrak U(\mathfrak g_\CC)$ by iterating the argument above.
\end{proof}
In Corollary~\ref{cor_differentiation_poincare_reihen},
notice that $X$ acts from the right, while the summation is over left translates. 
So for the action of $C_1$ and $C_2$ on the Poincar\'e series, we are reduced to compute the action on $H_\tau(g,s_1,s_2)$. 
 By Proposition~\ref{Prop_Casimir2_Operation}, it is enough to compute the actions
of $\trace(E_+E_-)$ and $\trace(E_+E_-E_+E_-)$.
Notice that $H_\tau$ is of the form $j(g,i)^{-\kappa}h(z)$, for a
function $h$ of the Siegel halfspace.
Define the differential operator
\begin{equation}
 D(h)=\sum_{a,b=1}^m\Bigl(\partial\bigl((y\bar\partial(h)y)_{ab}\bigr)\Bigr)_{ba}\:.
\end{equation}
\begin{prop}\label{Prop_Casimir_auf_f}
 Let $m$ and $\kappa$ be arbitrary. Let $h$ be a smooth function of the Siegel halfspace. Then:
\begin{equation*}
 \trace(E_+E_-)(j(g,i)^{-\kappa}h)\:=\:j(g,i)^{-\kappa}\bigl(
16\cdot D(h)
+8i(m+1-\kappa)\trace(y\bar\partial(h))\bigr)\:.
\end{equation*}
\end{prop}
\begin{proof}[Proof of Proposition~\ref{Prop_Casimir_auf_f}]
We make frequent use of Lemma~\ref{Eplusminus_auf_j}.
By the product rule, 
\begin{equation*}
 \trace(E_+E_-)\bigl(j(g,i)^{-\kappa}h\bigr)\:=\:
j(g,i)^{-\kappa}\trace(E_+E_-)(h)+
\trace\bigl(E_+(j(g,i)^{-\kappa})\cdot E_-(h)\bigr)\:,
\end{equation*}
where
\begin{equation*}
 \trace\bigl(E_+(j(g,i)^{-\kappa})E_-(h)\bigr)\:=\:-8i\kappa \cdot j(g,i)^{-\kappa}\trace\bigl(y\bar\partial(h)\bigr)\:.
\end{equation*}
The proposition follows once we have proved
\begin{equation*}
 \trace(E_+E_-)h(z)\:=\:8i(m+1)\trace (y\bar\partial(h(z))) +16D(h(z))\:.
\end{equation*}
But this follows summing up the terms
\begin{eqnarray*}
 (E_+)_{ab}(E_-)_{ba}(h)&=&
(-4i) (E_+)_{ab}\bigl(J'y\bar\partial(h)yJ\bigr)_{ba}\\
&=&
8i\sum_k\bigl(\bar JX^{(ab)}\bigr)_{ka}\bigl(y\bar\partial(h)yJ\bigr)_{kb}\\
&&+16\sum_{k,l}J_{ka}J_{lb}\bigl(\bar J'y\partial\bigl((y\bar\partial(h)y)_{kl}\bigr)y\bar J\bigr)_{ab}\\
&& 
+8i\sum_k\bigl(y\bar\partial(h)yJ\bigr)_{ak}\bigl(\bar JX^{(ab)}\bigr)_{kb}\:,
\end{eqnarray*}
where $ (\bar JX^{(ab)}\bigr)_{kb}\:=\:\frac{1}{2}(\delta_{bb}+\delta_{ab})\bar J_{ka}$.
The first and the last sum yield $4i(m+1)\trace(y\bar\partial(h))$ each, and the middle one yields $16D(h)$.
\end{proof}
 Calculations like those for Proposition~\ref{Prop_Casimir_auf_f}  yield:
\begin{prop}\label{prop_Casimir2_auf j_h}
Let $m$ and $\kappa$ be arbitrary. Let $h$ be a function of the Siegel halfspace. Then:
\begin{eqnarray*}
 && j(g,i)^\kappa\trace(E_+E_-E_+E_-)(j(g,i)^{-\kappa}h)\:=\:\\
&&\hspace*{1.8cm}
\Bigl[
8i(m+1)(m+1-\kappa)(m+2-2\kappa)\cdot\trace(y\bar\partial(h))\\
&&\hspace*{2cm}
+16(m+1)(3m+4-4\kappa)\cdot D(h)\\
&&\hspace*{2cm}
+4ij(g,i)^\kappa\trace(E_+E_-)(j(g,i)^{-\kappa}\trace(y\bar\partial(h))\\
&&\hspace*{2cm}
-32(m+1-\kappa)(m+2-2\kappa)\sum_{a,b}\left(\bar\partial\left((y\bar\partial(h)y)_{ab}\right)\right)_{ba}\\
&&\hspace*{2cm}
+16(4i)(m+2-2\kappa)\sum_{a,b,c}\left(\partial\left(\left( y\bar\partial\left((y\bar\partial(h)y)_{ab}
\right)\right)_{cb}\right)\right)_{ac}\\
&&\hspace*{2cm}
+16\cdot (8i)(m+1-\kappa)
\sum_{a,b,c}\left(\bar\partial\left(\left( y\partial\left((y\bar\partial(h)y)_{ab}
\right)\right)_{cb}\right)\right)_{ac}\\
&&\hspace*{2cm}
+16^2\sum_{a,b,c,d}\Bigl(\partial\bigl((y\bar\partial\bigl((\partial((y\bar\partial(h)y)_{ab}))_{bc}\bigr)y
)_{cd}\bigr)\Bigr)_{da}\Bigl.\Bigr]\:.
\end{eqnarray*}
\end{prop}
\section{Results for genus two and weight four}\label{sec_casimir_action_on_P_geschlecht_2}
Fix genus $m=2$ and weight $\kappa=4$.
Recall from Definition~\ref{definition_poincare}
\begin{equation}\label{eqn_def_s1_s2}
 s_1\:=\:\frac{v-2u-1}{2}\:,\:\quad s_2\:=\:\frac{u-2}{2}
\end{equation}
in this case.
According to Corollary~\ref{cor_konvergenzbereich_u_v}, the Poincar\'e series
\begin{equation*}
 P(g,u,v)\:=\:\sum_{\gamma\in \Gamma_\infty\backslash \Gamma}
\frac{\exp(2\pi i\trace(\tau \gamma\Hop z))}{j(\gamma g,i)^{\kappa}}\trace(\tau \im(\gamma\Hop z))^{\frac{v-2u-1}{2}}
\det (\im\gamma\Hop z)^{\frac{u-2}{2}}
\end{equation*}
converges in the cone defined by $\re u>2$, $\re v>5$. The point of holomorphicity  $(u,v)=(2,5)$ belongs to closure of $A$.
We give  results on the action of Casimir elements $C_1$ and $C_2$ on the Poincar\'e series  above
using Propositions~\ref{Prop_Casimir_auf_f}, \ref{prop_Casimir2_auf j_h}.
The  elementary but vast computations were verified by the computer algebra system Magma.
These results crucially depend on the weight $\kappa=4$.
\begin{eqnarray*}
 C_1 P(g,u,v) &=& 4(s_1^2+2s_1s_2+2s_2^2+2s_1+5s_2+8)  P(g,u,v)\\
&& -16\pi (s_1+s_2)P(g,u,v+2)\\
&& -8\det(\tau) s_1(s_1-1) P(g,u+2,v)\\
&& +32\pi\det(\tau) s_1P(g,u+2,v+2)\:,
\end{eqnarray*}
and 
\begin{eqnarray*}
 &&\hspace*{-5mm}C_2 P(g,u,v)\:=\: \\
 && 4(4s_1^4+16s_1^3s_2+24s_1^3+24s_1^2s_2^2+72s_1^2s_2+57s_1^2+16s_1s_2^3\\
  &&\hspace*{0.5cm}+72s_1s_2^2+114s_1s_2+46s_1+8s_2^4+40s_2^3+84s_2^2+51s_2+26) P(g,u,v)\\
  &&
  -256\pi^2\det(\tau) (s_1+s_1)(4s_1+2s_2+1) P(g,u+2,v+4)\\
  &&
  +32\pi\det(\tau) s_1(16s_1^2+36s_1s_2+30s_1+24s_2^2+40s_2+13)P(g,u+2,v+2)\\
  &&
  -8\det(\tau) s_1(s_1-1)(8s_1^2+24s_1s_2+28s_1+24s_2^2+60s_2+43)     P(g,u+2,v)\\
  &&
  +512\pi^2\det(\tau)^2 s_1(s_1-1)P(g,u+4,v+4)\\
  &&
  -256\pi\det(\tau)^2 s_1(s_1-1)(s_1-2) P(g,u+4,v+2)\\
  &&
  +32\det(\tau)^2 s_1(s_1-1)(s_1-2)(s_1-3) P(g,u+4,v)\\
  &&
  -16\pi(s_1+s_2)(8(s_1+s_2)(s_1+s_2+4)+37) P(g,u,v+2)\\
  &&
  +256\pi^2(s_1+s_2)(s_1+s_2+1)P(g,u,v+4)\:.
\end{eqnarray*}
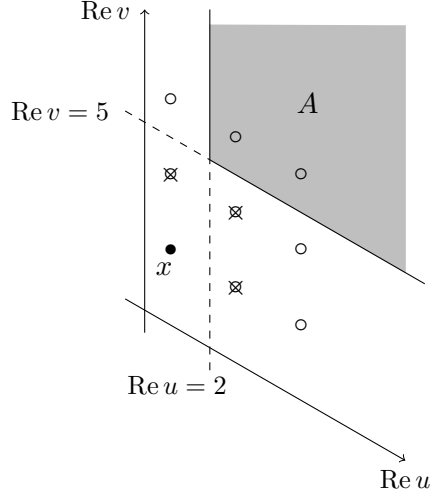
\begin{figure}
\begin{tikzpicture}

\pgftransformtriangle
{\pgfpoint{0pt}{0pt}}
{\pgfpoint{0.866pt}{-0.5pt}} 
{\pgfpoint{0pt}{1pt}}

\draw [->](-0.3,0) -- (4,0);
\draw (4,-0.25) node {\small $\re u$};
\draw [->](0,-0.3) -- (0,4);
\draw (-0.6,3.7) node {\small $\re v$};

\filldraw[lightgray] (1,4.3) -- (1,2.5) -- (4,2.5)--(4,5.79);
\draw (2.5,4) node {$A$};
\draw (1,2.5)--(1,4.5);
\draw [dashed] (1,-0.3) -- (1,2.5);
\draw (0.5,-0.75) node {\small $\re u=2$};
\draw (4.3,2.5) -- (1,2.5);
\draw [dashed] (-0.3,2.5) -- (1,2.5);
\draw (-1.3,2) node {\small $\re v=5$};

\draw (0.4,1) node {\small $\bullet$};
\draw (0.3,0.7) node {$x$};

\draw (0.4,2) node {$\times$};
\draw (1.4,1) node {$\times$};
\draw (1.4,2) node {$\times$};

\draw (0.4,3) node {$\circ$};
\draw (0.4,2) node {$\circ$};
\draw (2.4,1) node {$\circ$};
\draw (2.4,2) node {$\circ$};
\draw (1.4,1) node {$\circ$};
\draw (2.4,3) node {$\circ$};
\draw (1.4,2) node {$\circ$};
\draw (1.4,3) node {$\circ$};

\end{tikzpicture}
\caption{  \textit{A point $x=(u,v)$ and its shifts ``$\:\times$'' (respectively ``$\:\circ$'') corresponding to the shifted  
Poincar\'e series under $C_1$ (respectively $C_2)$.}}\label{bild_translate_unter_C1_undC_2}
\end{figure}
We are going to use the Casimir action to continue the Poincar\'e series analytically. 
If the Poincar\'e series produced by a Casimir operator $D$ had an area of convergence larger than $P(g,u,v)$ it was applied to, we could apply the resolvent of $D$ to them and would get
an analytic continuation of $P(g,u,v)$ to that area as long as the resolvent exists.
But both, $C_1$ and $C_2$, applied to the series produce several series which aren't of convergence
better than $P(g,u,v)$ itself at the same time.
This is much better understood by Figure~\ref{bild_translate_unter_C1_undC_2}.
There, the shifted series are marked according to their shifts (e.g. $P(g,u+2,v+4)$ corresponds to a shift by $(2,4)$).

But  defining
\begin{equation*}
 D_+(u)\::=\: \frac{1}{2}\bigl(C_1^2-C_2+11C_1-2(u^2-1)C_1+2(u^2-1)(u^2-4)\bigr)
\end{equation*}
we compute
\begin{eqnarray}
&&D_+(u)P(g,u,v)\:=\:\label{equ_dplus_kappa_2m}\label{Omega_groesser_angewandt}\\
&&\hspace*{1cm} 16\det(\tau)^2s_1(s_1-1)(s_1-2)(s_1-3)P(g,u+4,v)\nonumber\\
&&\hspace*{1cm}-128\pi\det(\tau)^2s_1(s_1-1)(s_1-2)P(g,u+4,v+2)\nonumber\\
&&\hspace*{1cm}+256\pi^2\det(\tau)^2s_1(s_1-1)P(g,u+4,v+4)\nonumber\\
&&\hspace*{1cm}+8\det(\tau)s_1(s_1-1)(u+1)(v+1)P(g,u+2,v)\nonumber\\
&&\hspace*{1cm}-182\det(\tau)s_1(s_1s_2+\frac{5}{6}s_1+\frac{4}{3}s_2^3+\frac{2}{3}s_2-2)P(g,u+2,v+2)\nonumber\\
&&\hspace*{1cm}+64\pi^2\det(\tau)(v-u-3)(u-3)P(g,u+2,v+4)\:.\nonumber
\end{eqnarray}
So there is an area~$B$ (see Figure~\ref{bild_Omega_+_translate}) such that the shifts
of $x=(u,v)$ by  $(4,0)$, $(4,2)$, $(4,4)$, $(2,0)$, $(2,2)$, and $(2,4)$
occurring in (\ref{Omega_groesser_angewandt}) belong to $A$ if and only if $x\in A\cup B$.
That is, the occurring shifted Poincar\'e series have concerted area of convergence
\begin{equation*}
 A\cup B\:=\:\{(u,v)\in\CC^2\mid \re u> 0 \textrm{ and } \re v>5\}\:.
\end{equation*}
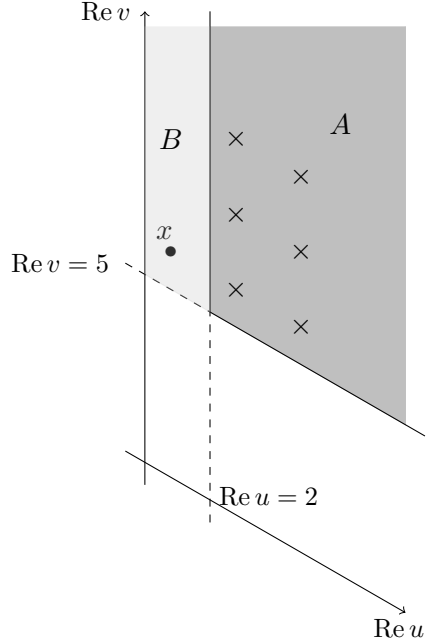
\begin{figure}
\begin{tikzpicture}
\pgftransformtriangle
{\pgfpoint{0pt}{0pt}}
{\pgfpoint{0.866pt}{-0.5pt}} 
{\pgfpoint{0pt}{1pt}}

\draw [->](-0.3,0) -- (4,0);
\draw (3.9,-0.25) node {\small $\re u$};
\draw [->](0,-0.3) -- (0,6);
\draw (-0.6,5.7) node {\small $\re v$};

\filldraw[lightgray] (1,6.3) -- (1,2.5) -- (4,2.5)--(4,7.8);
\draw (3,6) node {$A$};
\draw (1,2.5)--(1,6.5);
\draw [dashed] (1,-0.3) -- (1,2.5);
\draw (1.9,0.5) node {\small $\re u=2$};
\draw (4.3,2.5) -- (1,2.5);
\draw [dashed] (-0.3,2.5) -- (1,2.5);
\draw (-1.3,2) node {\small $\re v=5$};
\filldraw (0.4,3) node {\small $\bullet$};
\draw (0.3,3.2) node {$x$};

\draw (1.4,3) node {$\times$};
\draw (2.4,3) node {$\times$};
\draw (1.4,4) node {$\times$};
\draw (2.4,4) node {$\times$};
\draw (1.4,5) node {$\times$};
\draw (2.4,5) node {$\times$};

\filldraw[nearly transparent, lightgray] (0,5.8) -- (0,2.5) -- (1,2.5) -- (1,6.3);
\draw (0.4,4.5) node {$B$};
\end{tikzpicture}
\caption{  \textit{A point $x=(u,v)$ and its shifts corresponding to the shifted  Poincar\'e series under $D_+(u)$.
If $x$ belongs to $A\cup B$, then all the shifts belong to $A$.}}\label{bild_Omega_+_translate}
\end{figure}
Similarly, defining
\begin{equation*}
 D_-(v)\::=\: 2C_2-C_1^2-34C_1-2(v^2-9)C_1+(v^2-9)(v^2-1)\:,
\end{equation*}
and calculating
\begin{eqnarray}
 D_-(v)P(g,u,v)&=&+512\pi^2(s_1+s_2)(s_1+s_2+1)P(g,u,v+4)\label{equ_dminus_kappa_2m}\label{Omega_kleiner_angewandt}\\
&&+64\pi(s_1+s_2)(u-1)(v+1)P(g,u,v+2)\nonumber\\
&&+128\pi\det(\tau)s_1(s_1-3)(v+1)P(g,u+2,v+2)\nonumber\\
&&-1024\pi^2\det(\tau)(s_1+s_2)^2P(g,u+2,v+4)\:,\nonumber
\end{eqnarray}
for any point $x=(u,v)$ the shifts (see Figure~\ref{bild_Omega_-_translate}) by 
$(0,4)$, $(0,2)$, $(2,2)$ and $(2,4)$ belong to area~$A$ if and only if $x$ belongs to $A\cup C$ .
The four shifted Poincar\'e series occurring in (\ref{Omega_kleiner_angewandt}) have 
concerted area of convergence
\begin{equation*}
 A\cup C\:=\:\{(u,v)\in\CC^2\mid \re u>2 \textrm{ and } \re v>3\}\:.
\end{equation*}
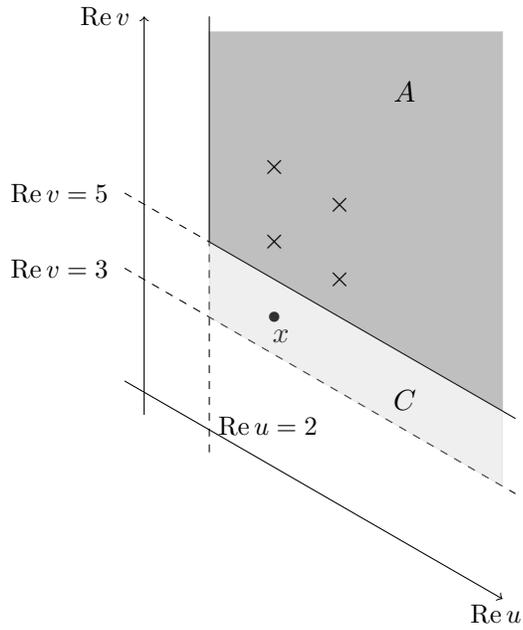
\begin{figure} 
\begin{tikzpicture}
\pgftransformtriangle
{\pgfpoint{0pt}{0pt}}
{\pgfpoint{0.866pt}{-0.5pt}} 
{\pgfpoint{0pt}{1pt}}

\draw [->](-0.3,0) -- (5.5,0);
\draw (5.4,-0.25) node {\small $\re u$};
\draw [->](0,-0.3) -- (0,5);
\draw (-0.6,4.7) node {\small $\re v$};

\filldraw[lightgray] (1,5.3) -- (1,2.5) -- (5.5,2.5)--(5.5,7.55);
\draw (4,6) node {$A$};
\draw (1,2.5)--(1,5.5);
\draw [dashed] (1,-0.3) -- (1,2.5);
\draw (1.9,0.5) node {\small $\re u=2$};
\draw (5.7,2.5) -- (1,2.5);
\draw [dashed] (-0.3,2.5) -- (1,2.5);
\draw (-1.3,2) node {\small $\re v=5$};

\draw (2,2) node {\small $\bullet$};
\draw (2.1,1.8) node {$x$};

\draw (2,3) node {$\times$};
\draw (2,4) node {$\times$};
\draw (3,3) node {$\times$};
\draw (3,4) node {$\times$};

\draw [dashed] (-0.3,1.5) -- (5.7,1.5);
\draw (-1.3,1) node {\small $\re v=3$};
\filldraw[nearly transparent, lightgray] (1,2.5) -- (1,1.5) -- (5.5,1.5) -- (5.5,2.5);
\draw (4,1.9) node  {$C$};
\end{tikzpicture}
\caption{  \textit{A point $x=(u,v)$ and its shifts corresponding to the shifted  Poincar\'e series under $D_-(v)$.
If $x$ belongs to $A\cup C$, then all the shifts belong to $A$.}}\label{bild_Omega_-_translate}
\end{figure}
Containing polynomials in the complex variables $u$ and $v$, the  operators $D_-(v)$ and $D_+(u)$ 
aren't selfadjoint any more but are very near to, as their real and imaginary parts are. 
The operators $D_+(u)$ and $D_-(v)$ originally were constructed such that only shifts 
in $u$ or $v$, respectively, occur. But their true shape is revealed by (\ref{wahre_gestalt_von_Dplusminus}).

The resolvents of $D_-(v)$ and $D_+(u)$ were studied in section~\ref{sec_resolvents} (Propositions~\ref{prop_discrete_spec}, 
\ref{prop_cont_spec_2-dim}, \ref{prop_cont_spec_1-dim}). 
As the 
resolvent $R_+(u)$ of $D_+(u)$ is a meromorphic function for $\frac{1}{2}<\re u$,  we  obtain the
meromorphic continuation of the Poincar\'e series to the halfstripe~$B'$ (see Figure~\ref{bild_fortsetzung_durch_resolventen}) by the $L^2$-function
\begin{equation*}
 P(\cdot,u,v)\:=\: R_+(u)\mathcal P_+(\cdot,u,v)\:.
\end{equation*}
Analogously, $R_-(v)\mathcal P_-(\cdot,u,v)$  establishes the meromorphic continuation to the halfstripe~$C$,
as the resolvent $R_-(v)$ is meromorphic on $3<\re v$.
Now both, $R_+(u)$ and  $R_-(v)$ give meromorphic continuation to area~$D$.
We can iterate this argument for $R_-(v)$ as long as this resolvent exists as a meromorphic function, i.e. as long as $\re v>1$ (area $E$). We get
%
\begin{figure}
\begin{tikzpicture}

\pgftransformtriangle
{\pgfpoint{0pt}{0pt}}
{\pgfpoint{0.866pt}{-0.5pt}} 
{\pgfpoint{0pt}{1pt}}

\draw [->](-0.3,0) -- (5,0);
\draw (4.9,-0.25) node {\small $\re u$};
\draw [->](0,-0.3) -- (0,5);
\draw (-0.6,4.7) node {\small $\re v$};

\filldraw[lightgray] (1,5.5) -- (1,2.5) -- (5.5,2.5)--(5.5,7.75);
\draw (4,5) node {$A$};
\draw (1,2.5)--(1,5.5);
\draw [dashed] (1,-0.9) -- (1,2.5);
\draw (1.8,-0.7) node {\small $\re u=2$};
\draw (5.5,2.5) -- (1,2.5);
\draw [dashed] (-0.3,2.5) -- (1,2.5);
\draw (-1.3,2) node {\small $\re v=5$};

\filldraw[nearly transparent, lightgray] (0.25,5.15) -- (0.25,2.5) -- (1,2.5) -- (1,5.5);
\draw (0.6,3.5) node {$B'$};

\draw [dashed] (-0.3,1.5) -- (5.5,1.5);
\draw (-1.3,1) node {\small $\re v=3$};
\filldraw[nearly transparent, lightgray] (1,2.5) -- (1,1.5) -- (5.5,1.5) -- (5.5,2.5);
\draw (4,1.9) node  {$C$};

\filldraw[semitransparent, lightgray] (0.25,1.5) -- (0.25,2.5) -- (1,2.5) -- (1,1.5);
\draw (0.6,1.9) node {$D$};

\draw [dashed] (-0.3,0.5) -- (5.5,0.5);
\draw (-1.3,0) node {\small $\re v=1$};
\draw [dashed] (0.25,-0.8) -- (0.25,5.5);
\draw (-0.5,-1.2) node {\small $\re u=\frac{1}{2}$};
\filldraw[nearly transparent, gray] (0.25,1.5) -- (0.25,0.5) -- (5.5,0.5) -- (5.5,1.5);
\draw (2,0.9) node {$E$};
\end{tikzpicture}

\caption{  \textit{Areas for continuation of the Poincar\'e series using the resolvents $R_+(u)$ and $R_-(v)$.}}
\label{bild_fortsetzung_durch_resolventen}
\end{figure}
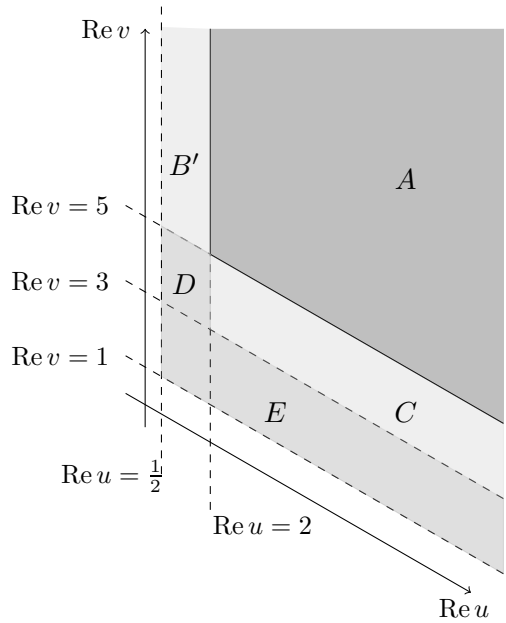
\begin{thm}\label{meromorpheFortsetzung}
 The Poincar\'e series $P(\cdot,u,v)$ admits a meromorphic continuation in $L^2(\Gamma\backslash G)$ to the cone 
\begin{equation*}
\{(u,v)\in \CC^2\mid \re u>\frac{1}{2},\: \re v>1\}\:.
\end{equation*}
The poles are contained in a finite number of lines $u=\mathop{const.}$ and $v=\mathop{const.}$.
\end{thm}
Theorem~\ref{meromorpheFortsetzung} is the best result we can get globally on $L^2(\Gamma\backslash G)$.
But by the means of section~\ref{sec_resolvents} we can  deduce exactly in which spectral components
a pole in question occurs.
According to the spectral decomposition (\ref{eqn_spectral_decompo}), 
we may decompose the Poincar\'e series within its domain of convergence into its orthogonal spectral components
\begin{equation*}
 P(\cdot,u,v)\:=\:  P_{\mathop{cont}}(\cdot,u,v)+\sum_{\Lambda}P_\Lambda(\cdot,u,v)\:.
\end{equation*}
\begin{cor}\label{cor_meromorphic_cont}
\begin{itemize}
 \item [(a)]
Any component $P_\Lambda(\cdot,u,v)$ of a discrete spectral component parametrized by $\Lambda$ has meromorphic continuation
to the entire plane $\CC^2$. Its singularities are at most poles and lie on the lines $u=u_0-2\mathbb N_0$ and 
$v=v_0-2\mathbb N_0$, where
$u_0$ and $v_0$ are zeros of $D_+(u,\Lambda)$ and $D_-(v,\Lambda)$ respectively.
\item[(b)]
The component $P_{\re\Lambda=0}(\cdot,u,v)$ of the $2$-dimensional continuous spectrum has analytic continuation to the cone
$\{(u,v)\mid\re u>0, \re v>0\}$.
\item[(c)]
Any component $P_{\gamma,c}(\cdot,u,v)$ of a $1$-dimensional spectral component pa\-ra\-me\-trized by $K_{\gamma}(c)$, where
$\gamma=\alpha_1$ or $\alpha_1+2\alpha_2$, has meromorphic continuation to the cone $\{(u,v)\mid \re u>0, \re v> c\}$.
 Its singularities are at most poles of order one and lie on $u=c$. Here $u=1$ occurs.
\item[(d)]
Any component $P_{\gamma,c}(\cdot,u,v)$ of a $1$-dimensional spectral component pa\-ra\-me\-trized by $K_{\gamma}(c)$, where
$\gamma=\alpha_2$ or $\alpha_1+\alpha_2$, has meromorphic continuation to the cone $\{(u,v)\mid\re u>\frac{ c}{2}, \re v>0\}$. 
Its singularities are at most poles of order one and lie on $v=c$.
\end{itemize}
\end{cor}
For the analytic continuation to the critical point $(u,v)=(2,5)$ we point out that we do not establish complete analyticity  there, but   boundedness along the line $s_1=0$.
This is exactly the line we need further on for the result on holomorphic projection.
\begin{thm}\label{thm_continuation_kappa_2m}
 Let $m=2$ and $\kappa=4$. 
The $L^2$-limit
\begin{equation*}
 P(\cdot,2,5)\::=\:\lim_{u\to 2} P(\cdot,u,2u+1)
\end{equation*}
exists. It is (formally) holomorphic and has got a $C^\infty$-representative. 
\end{thm}
\begin{proof}[Proof of Theorem~\ref{thm_continuation_kappa_2m}]
{\it  Spectral pole components in $(u,v)=(2,5)$.}
As the limit point in question is in the boundary of the convergence area, the possible poles are those of the resolvent $R_+(2)$
we use for  continuation. They are at most simple.
By Cor.~\ref{cor_meromorphic_cont}, 
the continued Poincar\'e series do not have  poles within the continuous spectrum, as 
$u=2>\frac{1}{2}$ and $v=5>1$.
So the only possible pole components are the discrete ones indexed by $\Lambda=(\Lambda_1,\Lambda_2)$, for which
\begin{equation*}
 D_+(2,\Lambda)\:=\: 0\:.
\end{equation*}
Up to Weyl conjugation, these are $\Lambda=(2,\Lambda_2)$.

{\it The limit $P(\cdot,2,5):=\lim_{(u,v)\to(2,5)}P(\cdot,u,v)$ exists in $L^2(\Gamma\backslash G)$.}
Along  the limit series
$s_1=0$, equivalently, $v=2u+1$, we find using 
\begin{equation*}
 D_+(2)\:=\: D_+(u)+(u^2-4)\bigl(C_1-(u^2-1)\bigr)\:,
\end{equation*}
and (\ref{equ_dplus_kappa_2m}) 
\begin{eqnarray*}
 D_+(2)P(\cdot,u,2u+1)&=&(u^2-4)(C_1-(u^2-1))P(\cdot,u,2u+1)\\
&& +64\pi^2\det(\tau)(u-2)(u-3)P(\cdot,u+2,2u+5)\:.
\end{eqnarray*}
Thus,
\begin{eqnarray*}
D_+(2)^nP(\cdot,u,2u+1)&=& (u^2-4)^n(C_1-(u^2-1))^nP(\cdot,u,2u+1)\\
&&+ (u-2)\mathcal P(\cdot,u)\:,
\end{eqnarray*}
where $\mathcal P(\cdot,u)$ is a symbol for a $\CC[u]$-linear combination of Poincar\'e series actually 
converging in $(u,v)=(2,5)$. Choosing $n=2$ greater than the maximal possible pole oder in $(u,v)=(2,5)$, we have
\begin{equation*}
 \lim_{u\to 2}\lvert\!\lvert (u^2-4)^n(C_1-(u^2-1))^nP(\cdot,u,2u+1)\rvert\!\rvert\:=\:0
\end{equation*}
as well as
\begin{equation*}
 \lim_{u\to 2}\lvert\!\lvert (u-2)\mathcal P(\cdot,u)\rvert\!\rvert\:=\:0\:.
\end{equation*}
Applying Schwarz' inequality, we deduce
\begin{equation*}
 \lim_{u\to 2}\lvert\!\lvert D_+(2)^n P(\cdot,u,2u+1)\rvert\!\rvert^2\:=\:0\:.
\end{equation*}
Written according to the spectral decomposition,
\begin{eqnarray*}
 0&=& \sum_\Lambda \lvert D_+(2,\Lambda)\rvert^{2n}\lim_{u\to 2}\lvert\!\lvert  P_\Lambda(\cdot,u,2u+1)\rvert\!\rvert^2\\
&&+\lim_{u\to 2}\lvert\!\lvert  D_+(2)^nP_{\mathop{cont}}(\cdot,u,2u+1)\rvert\!\rvert^2\:.
\end{eqnarray*}
For the component of the continuous spectrum we deduce
\begin{eqnarray*}
 \lim_{u\to 2}\lvert\!\lvert  P_{\mathop{cont}}(\cdot,u,2u+1)\rvert\!\rvert^2
&=& \lim_{u\to 2}\lvert\!\lvert  R_+(2)^nD_+(2)^nP_{\mathop{cont}}(\cdot,u,2u+1)\rvert\!\rvert^2\\
&\leq & \lvert\!\lvert R_+(2)\rvert\!\rvert^{2n}_{\mathop{cont}}\cdot\lim_{u\to 2}
\lvert\!\lvert  D_+(2)^n P_{\mathop{cont}}(\cdot,u,2u+1)\rvert\!\rvert^2
\:=\:0\:.
\end{eqnarray*}
While for any discrete spectral component the limit $\lim_{u\to 2}\lvert\!\lvert  P_\Lambda(\cdot,u,2u+1)\rvert\!\rvert^2$
exists and is nonzero only if $D_+(2,\Lambda)=0$. 
So the limit
\begin{equation*}
P(\cdot,2,5)\::=\:\lim_{u\to 2}  P(\cdot,u,2u+1)
\end{equation*}
exits as $L^2$-function. 

{\it The only non-vanishing spectral component is indexed by $\Lambda=(2,3)$.}
Along the line $v=2u+1$ we have
\begin{eqnarray*}
 D_-(2u+1) P(\cdot,u,2u+1) &=& 128\pi^2u(u-2)P(\cdot,u,2u+5) \\
 && +64\pi(u-2)(u-1)(u+1)P(\cdot,u,2u+3)\\
 &&-256\det(\tau) (u-2)^2P(\cdot,u+2,2u+5).
\end{eqnarray*}
Here the functions $P(\cdot,u,2u+5)$ and $P(\cdot,u,2u+5)$ are meromorphically continued to $\re u>1$ with at most simple poles in $u=2$ at the spectral zeros $\Lambda=(2,\Lambda_2)$
of $D_+(2,\Lambda)$.
We describe their residues with help of the simple Casimir operator $C_1$, which acts by the scalar $\Lambda_2^2-1$ on the component indexed by $\Lambda=(2,\Lambda_2)$.
On $s_1=0$ we have
\begin{eqnarray*}
 C_1P(\cdot,u,2u+1)&=& 4(2s_2^2+5s_2+8) P(\cdot,u,2u+1)-8\pi(u-2)P(\cdot,u,2u+3)\:.
\end{eqnarray*}
So as $L^2$-functions
\begin{equation*}
 (\Lambda_2^2-9)\lim_{u\to 2} P_\Lambda(\cdot,u,2u+1)\:=\:-8\pi\lim_{u\to 2}(u-2)P_\Lambda(\cdot,u,2u+3)\:.
\end{equation*}
While on $s_1=1$,
\begin{eqnarray*}
 C_1(\cdot,u,2u+3)&=& 4(2s_2^2+7s_2+11) P(\cdot,u,2u+3)\\
 &&-8\pi uP(\cdot,u,2u+5)\\
 &&32\pi\det(\tau) P(\cdot,u+2,2u+5)\:.
\end{eqnarray*}
Here the last series $P(\cdot,u+2,2u+5)$ converges in $u=2$, so
\begin{equation*}
 (\Lambda_2^2-45)\lim_{u\to 2}(u-2)P_\Lambda(\cdot,u,2u+3)\:=\: -16\pi\lim_{u\to 2}(u-2)P_\Lambda(\cdot,u,2u+5)\:.
\end{equation*}
So on the one hand
\begin{eqnarray*}
 \lim_{u\to2}D_-(2u+1)P_\Lambda(\cdot,u,2u+1)&=&
 128\pi^2\lim_{u\to2}u(u-2)P_\Lambda(\cdot,u,2u+5)\\
 &&+64\pi\lim_{u\to2}(u-2)(u^2-1)P_\Lambda(\cdot,u,2u+3)\\
 &=&2(\Lambda_2^2-9)(\Lambda_2^2-45)\lim_{u\to2}P_\Lambda(\cdot,u,2u+1)\:.
\end{eqnarray*}
But on the other hand, by (\ref{Dminus_in_koordinaten}) 
\begin{equation*}
 \lim_{u\to2}D_-(2u+1)P_\Lambda(\cdot,u,2u+1)=\left(\Lambda_2^2-9\right)\left(\Lambda_2^2-49\right)\lim_{u\to 2}P_\Lambda(\cdot, u,2u+1)\:.
\end{equation*}
So for $\lim_{u\to 2}P_\Lambda(\cdot, u,2u+1)$ not to vanish we must have $\Lambda_2^2=9$ or $\Lambda_2^2=65$.
Here $\Lambda=(2,\pm\sqrt{65})$ is not the infinitesimal character of a spectral component of $L^2(\Gamma\backslash \Sp_2(\RR))$. It  does not belong to the Eisenstein spectrum
as $\lvert\!\lvert \Lambda\rvert\!\rvert^2>\lvert\!\lvert \delta\rvert\!\rvert^2=5$. It does not belong to the (limits of) discrete series, as $\sqrt{65}$ is not integer.
The only remaining spectral component of $\lim_{u\to 2}P(\cdot, u,2u+1)$ is indexed by $\Lambda=(2,3)$
and belongs to the holomorphic limit of discrete series $\Pi_{(2,3)}$ of minimal $K$-type $(4,4)$ (Prop.~\ref{weissauers_bem_ueber_K-typ}).

{\it The $L^2$-limit $P=(\cdot,2,5):=\lim_{u\to 2}P_\Lambda(\cdot, u,2u+1)$ is formally holomorphic.}
We have seen $P(\cdot,2,5)=P_\Lambda(\cdot,2,5)$ for $\Lambda=(2,3)$.
As $C_1(\Lambda)=8$ and $C_1=\trace(E_+E_-)-\kappa m(m+1-\kappa)=\trace(E_+E_-)+8$, we have
\begin{equation*}
 0\:=\:\langle \trace(E_+E_-)P(\cdot,2,5),P(\cdot,2,5)\rangle\:=\:\sum_{i,j=1,2}\lvert\!\lvert (E_-)_{ij}P(\cdot,2,5)\rvert\!\rvert^2\:,
\end{equation*}
so $P(\cdot,2,5)$ is formally holomorphic.

{\it The $L^2$-limit $P(\cdot,2,5)$ is represented by
a $C^\infty$-function.} 
Adding an appropriate term to the Casimir $\trace(E_+E_-)+8$, we obtain an elliptic differential operator. 
$P(\cdot,2,5)$ is the solution of a corresponding elliptic differential equation with $C^\infty$-coefficients.
By regularity theory (cf.~\cite[Chapter 3.6.2]{aubin}), $P(\cdot,2,5)$ itself has a representative which  
belongs to $C^\infty(\Gamma\backslash G)$. 
Especially, it is pointwise defined.
\end{proof}
\section{Holomorphic projection}\label{holomorphe_projektion}
Let $L^2_\kappa(\Gamma\backslash \mathcal H)$ be the Hilbert space of functions on $\Gamma\backslash \mathcal H$ of weight 
$\kappa$ with scalar product 
\begin{equation*}
 \langle f,g\rangle \:=\:\int_{\mathcal F} f(z)\overline{g(z)}\det(y)^\kappa~dv_z\:.
\end{equation*}
Define the following Poincar\'e series of weight $\kappa$  on the Siegel halfplane $\mathcal H$
\begin{equation*}
 p_\tau(z,s)\::=\:\sum_{\gamma\in\Gamma_\infty\backslash \Gamma}e^{2\pi i\trace(\tau\gamma\cdot z)}
\frac{\det(\im\gamma\Hop z)^s}{j(\gamma,z)^\kappa}\:.
\end{equation*}
As
\begin{equation*}
 p_\tau(g\Hop i,s) \:=\: j(g,i)^\kappa \mathcal P_\tau(g,0,s)\:,
\end{equation*}
analytic properties are inherited:
\begin{cor}
For $\re s+\frac{\kappa}{2}>m$, the series $p_\tau(z,s)$ converges absolutely and locally uniformly in $s$, 
and uniformly on the Siegel fundamental domain $\mathcal F$. For $m=1$ or $m=2$ and  $\kappa\geq 2m$ the limit
\begin{equation*}\label{poincare_auf_H}
 p_\tau(\cdot)\::=\:\lim_{s\to 0}p_\tau(\cdot,s)\:\in\: L^2_\kappa(\Gamma\backslash \mathcal H)
\end{equation*}
exists and defines a holomorphic cuspform of weight $\kappa$ for $\Gamma$.
\end{cor}
\begin{proof}[Proof of Corollary~\ref{poincare_auf_H}]
This is the translation of  convergence  (Cor.~\ref{cor_konvergenzbereich_u_v}, \ref{corollar_genus_one})
 and continuation (Thm.~\ref{thm_continuation_kappa_2m}, \ref{thm_geschlecht_eins_holom}) 
on group level. 
The uniform convergence on $\F$ follows from 
\begin{equation*}
 \lvert p_\tau(z,s)\rvert\:\leq\: \det(y)^{-\frac{\kappa}{2}}c(s)\:,
\end{equation*}
and the fact that $\F$ is contained in a stripe  $\det y>c>0$. 
The holomorphicity is equivalent to the formal holomorphicity of the Poinca\'e series on group level
(see~\cite{weissauersLN}, p.~49). 
So $p_\tau$ is a holomorphic modular form of weight $\kappa\geq 2m$ which is square integrable.
It is either a cuspform or a lift of one (or a sum of both). For the latter, we get from~\cite[Theorem 2]{weissauer_paper}, 
 (see also~\cite[Lemma~9 and p.~108 Bemerkung]{weissauersLN}) the contradiction $\kappa<m$. 
So $p_\tau$ actually is a cuspform itself.
\end{proof}
Now we describe the projection to the holomorphic part of the spectrum in $L^2_\kappa(\Gamma\backslash \mathcal H)$.
For $m=2$ and $\kappa=4$, this is an application of our results.
It is well-known for high weight $\kappa>2m$ (\cite{panchishkin}) and for genus $m=1$ (\cite{gross-zagier}). 
We prove Theorem~\ref{satz_holomorphe_projektion} by the usual unfolding method.
It is nearly word-by-word \cite[IV.1]{gross-zagier}  or \cite[2.4]{panchishkin}, respectively.
We recall the gamma function of level $m$  
\begin{equation*}
 \Gamma_m(s)\:=\:\pi^{\frac{m(m-1)}{4}}\prod_{\nu=0}^{m-1}\Gamma(s-\frac{\nu}{2}) \:.
\end{equation*}
For $\re s>\frac{m-1}{2}$ there is a representation by the Euler integral
\begin{equation*}
 \Gamma_m(s)\:=\:\int_Y e^{-\trace(y)}\det(y)^{s-\frac{m+1}{2}}~dy\:,
\end{equation*}
where $Y=\{y=y'\in M_m(\RR)\mid y>0\}$. 
Define further $X:=\{x=x'\in M_m(\RR)\mid \lvert x_{jk}\rvert\leq \frac{1}{2}, 1\leq j,k\leq m\}$.
Recall (\cite[2.4]{panchishkin}) that a nonholomorphic modular form $F\in \tilde{\mathcal M}_\kappa(\Gamma)$ is of 
bounded growth if for all $\varepsilon>0$
\begin{equation*}
 \int_X\int_Y\lvert F(z)\rvert\det(y)^{\kappa-(m+1)}e^{-\varepsilon\trace(y)}~dy~dx <\infty\:.
\end{equation*}
If $\kappa\geq m+1$, equivalently, for an arbitrary $v\geq 0$ and all $\varepsilon>0$,
\begin{equation*}
 \int_X\int_Y\lvert F(z)\rvert\det(y)^{v}e^{-\varepsilon\trace(y)}~dy~dx \:<\:\infty\:.
\end{equation*}
\begin{thm}\label{satz_holomorphe_projektion}
Let $m=1$ or $m=2$ and let  $\kappa=2m$. 
Respectively, let $m$ be arbitrary and  let $\kappa>2m$.
 Let $F\in \tilde{\mathcal M}_\kappa(\Gamma)$ be of moderate growth and let
\begin{equation*}
 F(z)\:=\:\sum_{\tau}A_\tau(y)e^{2\pi i\trace(\tau x)}
\end{equation*}
be its Fourier expansion. That is, the sum is over all symmetic, half-integral $\tau$, 
and the coefficients $A_\tau$ are smooth functions on $Y$. 
Define for $\tau>0$
\begin{equation*}
 a(\tau)\::=\:c(m,\kappa)^{-1}\det(\tau)^{\kappa-\frac{m+1}{2}}\int_Y A_\tau(y)e^{-2\pi\trace(\tau y)}\det(y)^{\kappa-(m+1)}~dy\:,
\end{equation*}
where $c(m,\kappa):=(4\pi)^{m(\frac{m+1}{2}-\kappa)}\Gamma_m(\kappa-\frac{m+1}{2})$.
Then the function $\tilde F$ given by the Fourier expansion
\begin{equation*}
 \tilde F(z)\:=\:\sum_{\tau>0}a(\tau)e^{2\pi i\trace(\tau z)}
\end{equation*}
is a holomorphic cuspform of weight $\kappa$ for $\Gamma$ and for all $f\in\mathcal S_\kappa(\Gamma)$ we have 
\begin{equation*}
 \langle F,f\rangle\: =\: \langle \tilde F,f\rangle\:.
\end{equation*}
\end{thm}
\begin{proof}[Proof of Theorem~\ref{satz_holomorphe_projektion}]
$F$ being of bounded growth,
 the following integral exists  for $\re s\geq 0$
\begin{eqnarray*}
 &&\int_X\int_Y F(z)e^{-2\pi i\trace(\tau \bar z)}\det(y)^{s+\kappa}~dv_z\\
&&\hspace*{1cm}=\int_X\int_Y\sum_{\tilde \tau}A_{\tilde\tau}(y)e^{2\pi i\trace((\tilde\tau-\tau)x)}
e^{-2\pi\trace(\tau y)}\det(y)^{s+\kappa-(m+1)}~dx~dy\\
&&\hspace*{1cm}=\int_Y A_\tau(y)e^{-2\pi\trace(\tau y)}\det(y)^{s+\kappa-(m+1)}~dy\:.
\end{eqnarray*}
The value at $s=0$ of the last expression is by definition
\begin{equation*}
\int_Y A_\tau(y)e^{-2\pi\trace(\tau y)}\det(y)^{s+k-(m+1)}~dy\arrowvert_{s=0}\:=\:
c(m,\kappa)\det(y)^{\frac{m+1}{2}-\kappa}a(\tau)\:.
\end{equation*}
On the other hand, for $\re s+\frac{\kappa}{2}>m$, we get by unfolding
\begin{eqnarray*}
 \langle F, p_\tau(\cdot,\bar s)\rangle &=&
\int_{\mathcal F}F(z)\sum_{\gamma\in\Gamma_\infty\backslash \Gamma}e^{-2\pi i\trace(\tau\gamma\cdot\bar z)}
\frac{\det(\im\gamma\Hop z)^{s}}{\overline{j(\gamma,z)}^\kappa}\det(y)^\kappa~dv_z\\
&=&
\int_{\mathcal F}\sum_{\gamma\in\Gamma_\infty\backslash \Gamma}F(\gamma.z)e^{-2\pi i\trace(\tau\gamma\cdot\bar z)}
\det(\im\gamma\Hop z)^{s+\kappa}~dv_z\\
&=&
\int_X\int_Y F(z)e^{-2\pi i\trace(\tau\bar z)}\det(y)^{s+\kappa}~dv_z\:,
\end{eqnarray*}
where we may interchange  integration and summation because the Poincar\'e series converges uniformly in vertical stripes
 with $\det y>c>0$.
The two sides of this equation have analytic continuation to $s=0$, which must be equal:
\begin{equation*}
 \langle F,p_\tau\rangle\:=\:c(m,\kappa)\det(\tau)^{\frac{m+1}{2}-\kappa}a(\tau)\:.
\end{equation*}
Further, there exists a function
\begin{equation*}
 \tilde F(z)\:=\:\sum_{\tau>0}b(\tau)e^{2\pi i\trace(\tau z)}\:\in\:\mathcal S_\kappa(\Gamma)
\end{equation*}
which represents the antilinear mapping $\mathcal S_\kappa(\Gamma)\to\CC$, $f\mapsto\langle F,f\rangle$. 
That is, for all $f\in\mathcal S_\kappa(\Gamma)$,
\begin{equation*}
 \langle F,f\rangle\:=\:\langle \tilde F,f\rangle\:.
\end{equation*}
By the same calculation as for $F$, we find
\begin{eqnarray*}
 \langle \tilde F,p_\tau\rangle &=& b(\tau)\int_Ye^{-4\pi\trace(\tau y)}\det(y)^{\kappa-(m+1)}~dx\:dy\\
&=&b(\tau)c(m,\kappa)\det(\tau)^{\frac{m+1}{2}-\kappa}\:,
\end{eqnarray*}
which is valid for $\kappa>m$. As $p_\tau$ itself is a holomorphic cuspform, we must \linebreak have $a(\tau)=b(\tau)$.
\end{proof}
\section*{Appendix -- The case genus one}\label{section_geschlecht_eins}

The classical case of genus $m=1$  traces back to Selberg and Roelcke. Especially, the case $\kappa=m+1=2$ is part of~\cite{gross-zagier}.
The methods used in this paper apply to genus one and this case gives  a short outline of the ideas
omitting the technical requirements for genus two. 
The specific arguments for $m=1$ are due to Rainer Weissauer.

Consider the symplectic group of genus one, that is the special linear group $\SL_2(\RR)$. 
The spectral decomposition of 
 $L^2(\SL_2(\ZZ)\backslash \SL_2(\RR))$ with respect to the action of $\SL_2(\RR)$ by right translations 
is well known (see for example~\cite{lang}).
We use the common Langlands parameter $\Lambda$. So we have the Hilbert space orthogonal sum
\begin{equation*}
 L^2(\SL_2(\ZZ)\backslash \SL_2(\RR)) \:=\:  
L^2_{cont}(\SL_2(\ZZ)\backslash \SL_2(\RR))\:\bigoplus_\Lambda\: L^2_\Lambda(\SL_2(\ZZ)\backslash \SL_2(\RR))
\end{equation*}
of isotypical components $L^2_\Lambda(\SL_2(\ZZ)\backslash \SL_2(\RR))$ of irreducible
unitary representations $\Pi_\Lambda$ of infinitesimal character $\Lambda$, and the 
continuous spectrum \linebreak $L^2_{cont}(\SL_2(\ZZ)\backslash \SL_2(\RR))$. 
%
%
The  Casimir operator $C_1$  is given by (cf. Corollary~\ref{cor_casimir_operatoren})
\begin{equation*}
 C_1\:=\: E_+E_-+B^2-2B,
\end{equation*}
where
\begin{equation*}
 E_\pm \:=\:\begin{pmatrix} 1&\pm i\\\pm i&-1\end{pmatrix},\quad B\:=\:\frac{i}{2}\begin{pmatrix}0&-1\\1&0\end{pmatrix}
\end{equation*}
are the  generators of $\gg_\CC=\mathfrak{sl}_2(\RR)_\CC$.
The image of $C_1$ under the Harish-Chandra homomorphism is $\Lambda^2-1$.
On the continuous spectrum we have $\Lambda\in i\RR$, so
the Casimir $C_1$ is negative definite, being bounded from above by $-1$. 
While the discrete spectrum contains exactly one representation for which $\Lambda^2-1=0$ which contains the $K$-type $\kappa=2$, 
namely the discrete series representation of lowest $K$-type $-\kappa=-2$, and further representations which satisfy 
$\Lambda^2-1>c$ for a constant $c>0$.

Let  $\Gamma_\infty$ be the subgroup of translations in $\SL_2(\ZZ)$.
We define Poincar\'e series
of weight $\kappa=m+1=2$   by
\begin{equation*}
 P_\tau(g,s) \:=\:\sum_{\gamma\in \Gamma_\infty\backslash \SL_2(\ZZ)} h_\tau(\gamma g,s),
\end{equation*}
where $\tau$ is a positive integer and for 
\begin{equation*}
 g\:=\:\begin{pmatrix}
        y^{\frac{1}{2}}&xy^{-\frac{1}{2}}\\0&y^{-\frac{1}{2}}
       \end{pmatrix}
\begin{pmatrix}
 \cos\theta&\sin\theta\\-\sin\theta&\cos\theta
\end{pmatrix} \in \SL_2(\RR)
\end{equation*}
and $s\in \CC$ the function $h_\tau$ is t given by
\begin{equation*}
 h_\tau(g,s)\:=\:e^{2\pi i\tau z}\frac{y^{s-\frac{1}{2}}}{j(g,i)^\kappa} = 
e^{2\pi i\tau(x+iy)}y^{\frac{\kappa-1}{2}+s}e^{i\kappa\theta}.
\end{equation*}
The convergence result of Theorem~\ref{Konvergenzbereich} for genus $m=1$ is:
\begin{cor}\label{corollar_genus_one}
For genus $m=1$ and $\kappa=m+1$ the Poincar\'e series
\begin{equation*}
 P_\tau^1(g,s)\:=\:\sum_{\gamma\in\Gamma_\infty\backslash \Gamma} 
e^{2\pi i\tau\gamma\cdot z}\frac{(\im \gamma\Hop z)^{\re s-\frac{1}{2}}}{j(\gamma g,i)^2}
\end{equation*}
converges absolutely and belongs to $L^2(\Gamma\backslash\SL_2(\RR))$ in case $\re s>\frac{1}{2}$.
\end{cor}
We usually write
\begin{equation*}
 P_\tau(\cdot,s)\:=\:P_{\textrm{cont}}(\cdot,s)+\sum_{\Lambda} P_\Lambda(\cdot,s)  
\end{equation*}
according to the Hilbert space decomposition. 
Next we compute the action of $C_1$ on the Poincar\'e series. 
The constraint $\kappa=m+1$ implies that the share of $B^2-2B$ of the Casimir vanishes on the Poincar\'e series of 
weight $\kappa$.
We have
\begin{eqnarray}
 C_1 h_\tau(g,s)&=& 16 j(g,i)^\kappa\partial(y^2\bar\partial(e^{2\pi i\tau z}y^s))\nonumber\\
&=& 4(s^2-\frac{1}{4})h_\tau(g,s)-16\pi\tau (s-\frac{1}{2}) h_\tau(g,s+1)\:.\label{geschlecht_null_casimir_gleichung}
\end{eqnarray}
This  result can  be  verified easily using the Casimir operator in its classical shape (\cite{lang}, X,~\S2)
\begin{equation*}
 C_1 \:=\:4y^2\left(\partial_x^2+\partial_y^2\right)-4y\partial_x\partial_\theta
\end{equation*}
for functions on $\SL_2(\RR)$, or by a short evaluation of Proposition~\ref{Prop_Casimir_auf_f} in case $m=1$.
Notice that for functions of level $\kappa$, the operators $E_\pm$ are the shift $\pm 2$-operators (\cite{lang}, VI,~\S5).
The Casimir  $C_1$ corresponds  to the Laplacian on the Siegel halfplane up to a factor $-\frac{1}{4}$.
As we study a  representation by right translations, we may sum up all the left-translates of 
Equation~(\ref{geschlecht_null_casimir_gleichung})
to get
\begin{equation}\label{geschlecht_eins_casimir_gleichung}
 C_1 P_\tau(g,s)\:=\: 4(s^2-\frac{1}{4}) P_\tau(g,s) -16\pi\tau (s-\frac{1}{2}) P_\tau(g,s+1)\:,
\end{equation}
and the Poincar\'e series $P_\tau(g,s+1)$ on the right hand side actually converges  for $\re s>-\frac{1}{2}$.
\begin{thm}
 The resolvent $R(s)=\bigl(C_1-4(s^2-\frac{1}{4})\bigr)^{-1}$ exists for $\re s>0$ and $s$ outside a discrete set containing
$s=\frac{1}{2}$, where it has a simple poles.
\end{thm}
\begin{proof}
 We compute the image of $\bigl(C_1-4(s^2-\frac{1}{4})\bigr)$ under the Harish-Chandra homomorphism,
\begin{equation*}
 \Lambda\bigl(C_1-4(s^2-\frac{1}{4})\bigr) \:=\:\Lambda^2-4s^2\:.
\end{equation*}
For the continuous spectrum  the zeros $s=\pm\frac{\Lambda}{2}$ are contained in
$i\RR$. Restricting to the case $\re s>0$, the operator $\bigl(C_1-4(s^2-\frac{1}{4})\bigr)$ is therefore injective
on the continuous spectrum and we have
\begin{equation*}
 \lvert C_1-4(s^2-\frac{1}{4})\rvert \geq 4\lvert s\rvert^2\:.
\end{equation*}
In the discrete spectrum we actually have zeros $s=\pm\frac{\Lambda}{2}$. But as $\Lambda$ takes discrete values only, these
zeros are discrete as well. For example, $s=\frac{1}{2}$ is a zero for $\Lambda=\pm1$.
Apart from these zeros,
\begin{equation*}
 \lvert C_1-4(s^2-\frac{1}{4})\rvert \geq \min_\Lambda\{\lvert\Lambda^2-4s^2\rvert\}
\:=\::c(s)>0\:,
\end{equation*}
so the operator $\bigl(C_1-4(s^2-\frac{1}{4})\bigr)^{-1}$ exists 
if we assume $\re s>0$. As we have
\begin{equation*}
 \lvert\!\lvert \bigl(C_1-4(s^2-\frac{1}{4})\bigr)^{-1}\rvert\!\rvert\leq \max\{\frac{1}{c(s)},\frac{1}{4\lvert s\rvert^2}\}\:,
\end{equation*}
it is bounded.
As the singularities are simple zeros of $\Lambda\bigl(C_1-4(s^2-\frac{1}{4})\bigr)$, they are simple poles of $R(s)$.
\end{proof}
\begin{cor}
 The Poincar\'e series $P_\tau(\cdot,s)$ is continued meromorphically to the area $\re s>0$. It has a pole in $s=\frac{1}{2}$
of at most order one.
\end{cor}
\begin{proof}
By equation (\ref{geschlecht_eins_casimir_gleichung}),
$R(s)\bigl(-16\pi\tau s P_\tau(g,s+1)\bigr)$ defines the meromorphic continuation. As $R(s)$ has got a simple pole in 
$s=\frac{1}{2}$,
the resulting pole of the continuation is at most simple.
\end{proof}
%
For the limit $s\to \frac{1}{2}$ we 
apply the Casimir operator once more 
\begin{eqnarray}
 C_1^2 P_\tau(g,s) &=& 16(s^2-\frac{1}{4})^2P_\tau(g,s) - 64\pi\tau (s-\frac{1}{2})^2(s+\frac{1}{2})P_\tau(g,s+1)
\label{C_1^2_auf_P_Fall_g=1}\\
&&-16\pi (s-\frac{1}{2})C_1 P_\tau(g,s+1)\:.\nonumber
\end{eqnarray}
Notice that the limit of the $L^2$-norm of each single term on the right hand side  is zero as 
\begin{equation*}
\lim_{s\to \frac{1}{2}} \: \lvert\!\lvert (s-\frac{1}{2})^2P(\cdot, s)\rvert\!\rvert\:=\:0\quad
\textrm{ and }\quad
 \lim_{s\to \frac{1}{2}} \: \lvert\!\lvert (s-\frac{1}{2})P(\cdot, s+1)\rvert\!\rvert\:=\:0\:.
\end{equation*}
Applying Schwarz' inequality to the  norm of Equation~(\ref{C_1^2_auf_P_Fall_g=1}), we find
\begin{equation}\label{C^2_auf_P_gleich_0}
 \lim_{s\to \frac{1}{2}}\: \lvert\!\lvert C_1^2 P_\tau(\cdot,s)\rvert\!\rvert^2 \:=\: 0\:,
\end{equation}
or written according to the spectral decomposition,
\begin{equation*}
 0\:=\: \sum_\Lambda \lvert \Lambda(C_1)\rvert^2 \lim_{s\to \frac{1}{2}}\lvert\!\lvert P_\Lambda(\cdot,s)\rvert\!\rvert^2 
\:+\:
\lim_{s\to \frac{1}{2}}\lvert\!\lvert C_1^2 P_{\textrm{cont}}(\cdot,s)\rvert\!\rvert^2\:.
\end{equation*}
We deduce $\lim_{s\to \frac{1}{2}}\lvert\!\lvert C_1^2 P_{\textrm{cont}}(\cdot,s)\rvert\!\rvert=0$. 
As $C_1^2$ is positive definite on the 
continuous spectrum, thus
$\lim_{s\to \frac{1}{2}}\lvert\!\lvert  P_{\textrm{cont}}(\cdot,s)\rvert\!\rvert=0$. 
Further, $\lim_{s\to \frac{1}{2}}\lvert\!\lvert P_\Lambda(\cdot,s)\rvert\!\rvert =0$ as long as $\Lambda(C_1)\not=0$.
So the only nontrivial spectral component surviving the limit is that belonging to $\Lambda= 1$.
%
Accordingly, we have an analytic continuation of the Poincar\'e series $P_\tau(\cdot,s)$  to the point $s=\frac{1}{2}$ 
as $L^2$-function,
 and
the only nontrivially continued spectral component is that belonging
to the discrete series of minimal weight $\kappa=2$.
But now 
\begin{equation*}
 C_1 P_\tau(\cdot,\frac{1}{2})\:=\: (E_+E_-) P_\tau(\cdot,\frac{1}{2})\:=\: 0\:,
\end{equation*}
which implies
\begin{equation*}
 0\:=\:\langle (E_+E_-) P_\tau(\cdot,\frac{1}{2}),  P_\tau(\cdot,\frac{1}{2})\rangle\:=\:
-\langle E_- P_\tau(\cdot,\frac{1}{2}), E_- P_\tau(\cdot,\frac{1}{2})\rangle\:.
\end{equation*}
That is, $E_-P_\tau(\cdot,\frac{1}{2})=0$ in $L^2(\Gamma\backslash G)$, that is $P_\tau(\cdot,\frac{1}{2})$ is formally holomorphic. 
Now $C_1$ can easily be changed to an elliptic differential operator by adding appropriate terms.
So the limit $P_\tau(\cdot,\frac{1}{2})$ is the solution of an elliptic differential equation with $C^\infty$-coefficients.
As local regularity of the coefficients is inherited by the solution, $P_\tau(\cdot,\frac{1}{2})$ itself is $C^\infty$.
We have proved:
\begin{thm}\label{thm_geschlecht_eins_holom}
 The limit
\begin{equation*}
 P_\tau(\cdot,\frac{1}{2})\::=\:\lim_{s\to \frac{1}{2}} P_\tau(\cdot,s)
\end{equation*}
exists in $ L^2(\SL_2(\ZZ)\backslash \SL_2(\RR))$. Its only nonzero spectral component belongs to the discrete series 
$\Pi_{\Lambda=1}$.
Thus, $P_\tau(\cdot,\frac{1}{2})$ is a zero of the Casimir operator.
It is represented by a $C^\infty(\Gamma\backslash G)$-function, which  is actually holomorphic.
\end{thm}

\end{document}